\documentclass[a4paper]{article}

\usepackage{microtype} 
\usepackage{lmodern} 

\usepackage{amsmath,amsfonts,amssymb,amsopn}
\usepackage{graphicx}
\usepackage{epstopdf} 
\usepackage{flafter} 
\usepackage{amsthm} 

\usepackage{subcaption} 
\usepackage{placeins} 
\usepackage{siunitx} 
\usepackage{algorithm, algpseudocode} 
\usepackage{xcolor} 

\usepackage[hidelinks]{hyperref} 

\usepackage{tikz, pgfplots} 
\usetikzlibrary{backgrounds}

\newlength\figureheight
\newlength\figurewidth

\newcommand{\sethw}[2]{
	\setlength\figureheight{#1\textwidth}
	\setlength\figurewidth{#2\textwidth}	
}

\usepackage{pgf}

\newcommand{\bs}[1]{\boldsymbol{#1}} 

\renewcommand{\d}[1]{\ensuremath{\operatorname{d}\!{#1}}}

\newcommand{\diff}[2]{\ensuremath{\frac{\d{#1}}{\d{#2}}}}
\newcommand{\diffp}[2]{\ensuremath{\frac{\partial{#1}}{\partial{#2}}}}
\newcommand{\boundary}[1]{\partial\!{#1}}
\newcommand{\hess}[1]{ \mbox{Hess } #1}

\newcommand{\mean}[1]{\ensuremath{\mathbb{E}\mathopen{}\left[{#1}\right]\mathclose{}}}
\newcommand{\var}[1]{\ensuremath{\mathbb{V}\mathopen{}\left[{#1}\right]\mathclose{}}}

\newcommand{\norm}[1]{\ensuremath{\|{#1}\|}}

\newcommand{\ip}[2]{(#1,#2)} 
\newcommand{\und}{\cdot} 

\newtheorem*{remark}{Remark}

\newtheorem{theorem}{Theorem}
\newtheorem{corollary}{Corollary}
\newtheorem{assumption}{Assumption}

\newcommand{\hl}[1]{\emph{#1}}

\usepackage[margin=1.3in]{geometry} 

\title{MG/OPT and MLMC for Robust Optimization of PDE\MakeLowercase{s}}
\author{A. Van Barel \and S. Vandewalle}

\begin{document}
	
\maketitle

\begin{abstract}
  An algorithm is proposed to solve robust control problems constrained by partial differential equations with uncertain coefficients, based on the so-called MG/OPT framework.
  The levels in the MG/OPT hierarchy correspond to discretization levels of the PDE, as usual. For stochastic problems, the relevant quantities (such as the gradient) contain expected value operators on each of these levels. They are estimated using a multilevel Monte Carlo method, the specifics of which depend on the MG/OPT level. Each of the optimization levels then contains multiple underlying multilevel Monte Carlo levels. The MG/OPT hierarchy allows the algorithm to exploit the structure inherent in the PDE, speeding up the convergence to the optimum. In contrast, the multilevel Monte Carlo hierarchy exists to exploit structure present in the stochastic dimensions of the problem. A statement about the asymptotic cost of the algorithm is proven, and some additional properties are discussed. 
  The performance of the algorithm is numerically investigated for three test cases. 
  A reduction in the number of samples required on expensive levels and therefore in computation time can be observed.
\end{abstract}




\section{Introduction}
This paper addresses the optimization of systems constrained by partial differential equations (PDEs) with uncertain parameters. An optimizer is considered robust if it performs well across a wide set of realizations of those parameters. Due to the uncertainty, a given deterministic control input yields a stochastic cost. The optimizer's performance is usually captured by introducing a \emph{risk measure}, which assesses the desirability of the original cost's distribution. The cost functional obtained by applying the risk measure to that original cost, is then deterministic again.
Many different choices for the risk measure exist; see, e.g., \cite{rockafellar2013fundamental, kouri2016risk}. 
The expected value operator, which we employ in this paper, is the most straightforward one.
Nevertheless, other risk measures may also be compatible with the method described in this paper. A necessary condition is that the gradient of the resulting cost exists and can be estimated to arbitrary precision using a Monte Carlo method. Not all formulations for incorporating the uncertainties are obtained as described above. Some other common approaches are described in, e.g., \cite{Ali2016, Borzi2012, Borzi2010, lee2017comparison}. However, in this paper, the term \emph{robust control problem} specifically indicates the optimization of the expected value of some cost.

The robust control problem has been addressed previously using various methods. The stochastic Galerkin and stochastic collocation methods \cite{borzi2009VWmultigrid, tiesler2012stochastic, Rosseel2012} are most suitable for lower-dimensional stochastic spaces, since the number of collocation points increases rapidly with the dimension. A proper orthogonal decomposition method has been investigated in \cite{borzi2011pod}, and reduced basis methods in, e.g., \cite{chen2014weighted}. Recently, quasi-Monte Carlo methods have been investigated in this context \cite{guth2019quasi}.
Many techniques apply some existing (multigrid) solver to a sampled version of the problem. This has the drawback that the sampling is then independent of the solver's progress, e.g., in a multigrid solver, the sampling is the same regardless of the multigrid level.
Furthermore, the sampling scheme could be adapted to the current iteration or, if applicable, optimization step. These ideas are used in, e.g., \cite{kouri2013trust}. Stochastic gradient methods have been thoroughly investigated for elliptic problems in the work of Martin, Krumscheid and Nobile \cite{martin2018analysis} and others \cite{geiersbach2019stochastic}. 
Finally, MG/OPT techniques \cite{lewis2000multigrid, lewis2005} have been used for robust optimization. MG/OPT is a multigrid approach to optimization problems. It requires coarser optimization problems to be defined at each of the MG/OPT levels. They serve to cheaply calculate updates for the finer levels recursively. A major inspiration for our work was the method due to Kouri \cite{Kouri2014}. There, the `coarser' levels correspond to taking a smaller number of sample points in the stochastic space. 

In a previous paper, we showed that the multilevel Monte Carlo (MLMC) method \cite{cliffe2011, giles2015} can be used to efficiently generate a gradient (or Hessian vector product) for problems constrained by a PDE with uncertain coefficients \cite{vanbarel2019robust}.
These quantities are normally returned at the finest level. Unless the gradient (or Hessian) based optimization algorithm is somehow made aware of the details of the MLMC calculation, optimization steps can then only be done on the finest level. For some problems, this may result in a rather slow convergence to the optimum. To remedy this, we propose to combine this approach with the MG/OPT framework. We construct our coarser optimization problems by not only reducing the number of samples, but also by coarsening the discretization of the spatial and temporal dimensions. Each MG/OPT level thus has its own MLMC hierarchy. Coarser MG/OPT levels retain only the coarser MLMC levels, and take fewer samples on those levels. A qualitative comparison between the proposed method and the existing methods is deferred to \S\ref{sec:qualcomp}. This allows some aspects that are of interest to be described in more detail first. In the final parts of the paper, a practical algorithm with stopping conditions is given. Under some assumptions, we can prove a theorem about the convergence rate of that algorithm.

The paper is structured as follows. The notation to describe the robust control problem is introduced in \S\ref{sec:problem_formulation} and the corresponding optimality conditions are described in general terms in \S\ref{sec:optconds}. Next, \S\ref{sec:MLMC} describes the MLMC method and its properties in this context. 
MG/OPT and some of its properties are recalled in \S\ref{sec:MG/OPT}. Furthermore, details such as the precise number of samples to take at each MG/OPT and MLMC level and the specifics of the line searches and the underlying optimization algorithms are also discussed. \S\ref{sec:qualcomp} provides a qualitative comparison between the method of Kouri \cite{Kouri2014}, previous multigrid based methods, e.g., \cite{Rosseel2012}, the MLMC method \cite{vanbarel2019robust} and the method described in this paper. The practical algorithm is detailed in \S\ref{sec:robopt}. The convergence rate of the cost to obtain a gradient norm of a given accuracy is investigated and some other details are discussed. 
Finally, \S\ref{sec:numresults} provides numerical evidence of the method's performance for three fairly different model problems. These are the standard elliptic control problem with lognormal diffusion coefficient, the Dirichlet to Neumann map, and a control problem constrained by Burgers' equation. Some concluding remarks can be found in \S\ref{sec:conclusion}.

\section{Problem formulation}
\label{sec:problem_formulation}
\newcommand{\rJ}{\ensuremath{\hat{J}}}
\newcommand{\stvar}{\ensuremath{k}} 
Let $(\Omega, \mathcal{A}, \mu)$ denote a complete probability space. The sample space $\Omega$ contains all possible realizations $\omega$ of the random influence. Its dimension is the stochastic dimension of the problem and may be infinite. $\mathcal{A}$ is the set of all events (subsets of $\Omega$) and $\mu$ is a measure that maps events in $\mathcal{A}$ to probabilities in $[0,1]$. 
The expected value operator and the variance operator of a stochastic variable $\stvar$ are denoted as follows
\begin{align*}\mean{\stvar} \triangleq \int_\Omega \stvar \d{\mu(\omega)}, \quad\; \var{\stvar} \triangleq \mean{(\stvar - \mean{\stvar})^2} = \mean{\stvar^2} - \mean{\stvar}^2.\end{align*}

Let $U$ and $Y$ denote Hilbert spaces in which, respectively, the control and state reside. The PDE is then formally represented by $c(y,u) = 0$ with $c: Y \times U \rightarrow Y^*$. 
Due to uncertainties in the parameters of the PDE, the state $y$ is stochastic. We assume that $Y = Z \otimes L^2(\Omega)$ with $Z$ a Hilbert space over the physical domain(s) of individual (deterministic) realizations of the state, and $\otimes$ denoting the tensor product. The control $u$, however, is deterministic in this paper, corresponding to the situation where its choice must be made without incorporating information about the specific realization of the stochastic parameters. 
The constraint $c$ is understood to represent the PDE for ($\mu$-almost) all realizations $\omega$ of the stochastic influence, i.e., $c(y,u) = 0$ means $c_\omega(y_\omega, u) = 0$ $\mu$-a.e.~in $\Omega$, with $c_\omega : Z \times U \rightarrow Z^*$ representing one realization of the PDE. We require that $c(y,u) = 0$ can be uniquely solved for all $u \in U$, defining a mapping $S: U \rightarrow Y : u \mapsto y = S(u)$. This requirement puts a constraint on which variables can be chosen as control variables. We further assume $U,Z\subseteq L^2$ and therefore $Y \subseteq L^2 \otimes L^2(\Omega)$.
In this paper we solve robust control problems for tracking type cost functionals, such as for example
\begin{equation}
\min_{u \in U, y \in Y} J(y,u) = \frac{1}{2} \mean{\norm{y-z}^2} + \frac{\alpha}{2}\norm{u}^2 \quad \text{s.t.}\quad c(y,u)=0,
\label{eq:robust_problem}
\end{equation}
where $\|\und\|$ denotes the $L^2$-norm (over the physical domain(s) of $y$ or $u$) and $z \in L^2$ denotes some deterministic target function.
One can eliminate the constraint using $S$ to obtain the so-called reduced problem 
\begin{equation}
\min_{u \in U} \rJ(u) \;\text{ with }\; \rJ(u) \triangleq J(S(u),u) = \frac{1}{2} \mean{\norm{S(u)-z}^2} + \frac{\alpha}{2}\norm{u}^2.
\label{eq:reduced_robust_problem}
\end{equation}

In any case, a cost functional $J: Y \times U \rightarrow \mathbb{R}$ must be such that the result is always deterministic. In this paper, this is achieved simply by employing the expected value operator. A number of other risk measures, e.g., risk measures including a penalty on the variance of the state $y$, can also be solved using the methods described in this paper \cite{vanbarel2019robust}. 

\section{Optimality conditions}
\label{sec:optconds}
\newcommand{\lm}{p} 
Let $\ip{\und}{\und}_{U}$, $\ip{\und}{\und}_{Z}$ and $\ip{\und}{\und}_{Y}$ denote the inner products in $U$, $Z$, and $Y$, respectively. Since $Y = Z \otimes L^2(\Omega)$, we naturally take $\ip{\und}{\und}_{Y} = \mean{\ip{\und}{\und}_{Z}}$. In the remainder of this text, elements of the dual space under those inner products are assumed to be pulled back to the primal space. 
The optimality conditions can be formally derived by introducing a Lagrange multiplier $\lm \in Y$, defining the Lagrangian
$$ \mathcal{L}(y,u,\lm) \triangleq J(y,u) + \ip{\lm}{c(y,u)}_{Y}$$
and setting its derivatives to $\lm,y$, and $u$ to zero, yielding
\begin{equation}
\left\{
\begin{aligned}
\nabla_\lm \mathcal{L} &= c(y,u) = 0 \vphantom{\diffp{c}{y}}, \\
\nabla_y \mathcal{L} &= \nabla_y J + \Big(\diffp{c}{y}\Big)^*[\lm] = 0, \\
\nabla_u \mathcal{L} &= \nabla_u J + \Big(\diffp{c}{u}\Big)^*[\lm] = 0.
\end{aligned}
\right.
\label{eq:optcond_lagr_general}
\end{equation}
The symbol $^*$ is used to denote the adjoint of a linear operator. These equations are referred to as the constraint equation, the adjoint equation, and the optimality condition, in the order as stated 
above.
Since for any $\tilde{u} \in U$,
$$
\diffp{\ip{p}{c}_{Y}}{u}[\tilde{u}] = \ip{p}{\diffp{c}{u}[\tilde{u}]}_{Y} 
	= \mean{ \ip{p_\omega}{\diffp{c_\omega}{u}[\tilde{u}]}_{Z} }
	= \ip{\mean{\Big(\diffp{c_\omega}{u}\Big)^* [p_\omega]}}{\tilde{u}}_{U},
$$
the optimality condition can be rewritten as 
\begin{equation}
\label{eq:1}
\nabla_u J + \mean{\Big(\diffp{c_\omega}{u}\Big)^*[\lm_\omega]} = 0.
\end{equation}
The left hand side in (\ref{eq:1}) expresses the gradient of the reduced cost functional w.r.t. $\ip{\und}{\und}_U$, i.e., $\nabla \rJ$.
Optimality thus requires the reduced gradient to be zero.
Taking $U = L^2$, the equations for the reduced gradient for the cost function in the example (\ref{eq:robust_problem}) simplify to 
\begin{equation}
\left\{
\begin{aligned}
c(y,u) &= 0 \vphantom{\diffp{c}{y}},\\
-\Big(\diffp{c}{y}\Big)^*[\lm] &= y-z,
\\
\nabla \rJ(u) &= \alpha u + \mean{\Big(\diffp{c_\omega}{u}\Big)^*[\lm_\omega]}.
\end{aligned}
\right.
\label{eq:optcond_lagr}
\end{equation}
\begin{assumption}
	The derivatives (Fr\'echet derivatives) in the formulas above exist. Furthermore, the $\omega$-dependent quantities  can be evaluated a.s., i.e., for almost all $\omega \in \Omega$.
\end{assumption}

\section{Multilevel Monte Carlo}
\label{sec:MLMC}
\newcommand{\chlvl}[2]{\smash{I_{#1}^{#2}}}
\newcommand{\Lmax}{L}
Since we wish to be able to deal with high-dimensional or irregularly behaving stochastic spaces, we propose to evaluate the expected value in (\ref{eq:optcond_lagr}) using an MLMC method. For a given control input $u \in U$, the so-called \hl{quantity of interest} is then $Q \in U \otimes L^2(\Omega)$, where $Q(\omega) \triangleq \left(\diffp{c_\omega}{u}\right)^*[\lm_\omega]$. Its expected value is in $U$. The dependence of $Q$ on $u$ is omitted in the notation in this part of the text.
Because Q depends on the solution of PDEs, one cannot generate exact samples $Q(\omega) \in U$. Instead, we consider a hierarchy of spaces $U_0 \subset U_1 \subset \cdots 
\subset U$ in which we can generate samples $Q_\ell(\omega) \in U_\ell$ of an approximation $Q_\ell$ of $Q$. The subscript $\ell$ is a measure of the accuracy of the approximation, which in our case corresponds to the level of discretization that the PDEs are solved on.

Each of the spaces $U_\ell$ are endowed with inner products $\ip{\und}{\und}_\ell$.
Prolongation operators $\chlvl{\ell}{\ell + 1} : U_\ell \rightarrow U_{\ell+1}$ and restriction operators $\chlvl{\ell+1}{\ell} : U_{\ell+1} \rightarrow U_{\ell}$ are used to map between the discretization levels.
Let $\chlvl{\ell}{\ell'} \triangleq \chlvl{\ell'-1}{\ell'}\chlvl{\ell'-2}{\ell'-1} \cdots \chlvl{\ell}{\ell+1}$ if $\ell < \ell'$ and the analogue if $\ell > \ell'$. We require that \begin{equation}
\forall u \in U_\ell, u'\in U_{\ell'} : \ip{\chlvl{\ell}{\ell'}u}{u'}_{\ell'} = \ip{u}{\chlvl{\ell'}{\ell}u'}_{\ell}.
\label{eq:chlvl_adjoint}
\end{equation}
Furthermore, let $\chlvl{\ell}{} : U_\ell \rightarrow U$ denote the inclusion of $U_\ell$ into $U$.

Define $a \lesssim b$ as $a \leq cb$ with $c>0$ some constant independent of $a$ and $b$, and $a \eqsim b$ as $a \lesssim b$ and $b \lesssim a$. The following is assumed to hold.
\begin{assumption}
	The numerical scheme has a weak order of convergence $\rho$, i.e.,
	$\norm{\mean{\chlvl{\ell}{}Q_\ell - Q}} \newline \lesssim 2^{-\rho\ell} \label{eq:disc_error}$
	with $\rho > 0$.
	The computational cost for a single sample, denoted $\mathcal{C}_\ell$, satisfies
	$\mathcal{C}_\ell \lesssim 2^{\kappa\ell} \label{eq:sample_cost}$
	for some constant $\kappa$. 
	\label{as:rhokappa}
\end{assumption}
\noindent Both $\rho$ and $\kappa$ depend on the algorithm employed to solve the PDE.

\newcommand{\YMC}{\ensuremath{\hat{Y}^\textup{MC}}}
\newcommand{\QMLMC}{\ensuremath{\hat{Q}^\textup{ML}_{\vec{n}}}}
\newcommand{\nvec}{\{n_\ell\}}
The MLMC method \cite{cliffe2011, giles2015} recursively estimates an expected value on a finer level as an expected value on a coarser level (acting as a control variate) combined with a corrective term. This leads to a telescopic sum decomposition
\begin{equation}
\mean{Q_{L}} = \chlvl{0}{L}\mean{Q_{0}} + \sum\limits_{\ell=1}^{L}\chlvl{\ell}{L}\mean{Q_{\ell} - \chlvl{\ell-1}{\ell}Q_{{\ell-1}}} = \sum\limits_{\ell=0}^{L}\chlvl{\ell}{L}\mean{Y_\ell}
\label{eq:telescopic_sum}
\end{equation}
where $Y_\ell \triangleq Q_{\ell} - \chlvl{\ell-1}{\ell}Q_{{\ell-1}}$ if $\ell > 0$, and $\chlvl{-1}{0}$ and $Q_{-1}$ defined in some way such that $\chlvl{-1}{0}Q_{-1} = 0 \in U_0$.
This reduces the cost by avoiding the direct estimation of $\mean{Q_{L}}$ on the finest level.
On level $\ell$, $\mean{Y_\ell}$ is estimated using the ordinary Monte Carlo (MC) method with $n_\ell$ samples, yielding
\begin{equation}
\YMC_{\ell, n_\ell} \triangleq \frac{1}{n_\ell}\sum\limits_{i=1}^{n_\ell} Y_\ell(\omega_i) = \frac{1}{n_\ell}\sum\limits_{i=1}^{n_\ell} \Big(Q_{\ell}(\omega_i) - \chlvl{\ell-1}{\ell}Q_{{\ell-1}}(\omega_i)\Big).\label{eq:MC_def1Y}
\end{equation}
It is important to use the same stochastic realization $\omega_i$ on both levels for each sample of $Y_\ell$ to ensure a high correlation. The MLMC estimator is then defined as
\begin{equation}
\QMLMC \triangleq \sum\limits_{\ell=0}^{L} \chlvl{\ell}{L}\YMC_{\ell,n_\ell} \label{eq:MLMC_def2Q}
\end{equation}
with $\vec{n} = \nvec_{\ell = 0}^L$. All expectations $\mean{Y_\ell}$ in (\ref{eq:telescopic_sum}) should be estimated independently, ensuring that (\ref{eq:MLMC_def2Q}) is an unbiased estimator for $Q_L$, i.e., $\smash{\mathbb{E}[{\QMLMC}] = \mean{Q_{L}}}$.

\subsection{Mean square error}
\label{sec:MLMC/MSE}
The mean square error (MSE) of $\chlvl{L}{}\QMLMC$ as an estimator for $\mean{Q}$ can be characterized, see \cite{cliffe2011}, as follows
\begin{align}
\mean{\norm{\chlvl{L}{}\QMLMC - \mean{Q}}^2} 
&= \mean{\norm{\chlvl{L}{}\QMLMC - \chlvl{L}{}\mean{\QMLMC} + \chlvl{L}{}\mean{\QMLMC} - \mean{Q}}^2} \nonumber \\
&=\underbrace{
\mean{\norm{\chlvl{L}{}\QMLMC -\chlvl{L}{}\mean{\QMLMC}}^2}
}_\text{Stochastic error}
+
\underbrace{
\norm{\chlvl{L}{}\mean{\QMLMC} - \mean{Q}}^2
}_\text{Bias}
\label{eq:MLMC_error}
\end{align}
Since (\ref{eq:MLMC_def2Q}) is an unbiased estimator for $Q_L$, the bias term can be expressed as
\begin{equation}
\norm{\mean{\chlvl{L}{}Q_L - Q}}^2.
\label{eq:MLMC_bias}
\end{equation} 
This term is due to the discretization error. It can be decreased by solving the PDE using a finer discretization, i.e., by increasing $L$. The stochastic error can be further worked out as
\newcommand{\physu}{{D_U}}
\begin{align}
\mean{\norm{\chlvl{L}{}\QMLMC - \chlvl{L}{}\mean{\QMLMC}}^2}
&= \mean{\norm{
	\sum\limits_{\ell=0}^{L} (
	\chlvl{\ell}{}\YMC_{\ell,n_\ell} - \chlvl{\ell}{}\mean{\YMC_{\ell,n_\ell}}
	)
}^2} \nonumber\\
&= \sum\limits_{\ell=0}^{L}
	\mean{\norm{
	\chlvl{\ell}{}\YMC_{\ell,n_\ell} - \chlvl{\ell}{}\mean{\YMC_{\ell,n_\ell}}
	}^2}
 \nonumber\\
&= \sum\limits_{\ell=0}^{L}
	\int_\physu
	\var{ \chlvl{\ell}{}\YMC_{\ell,n_\ell} }
	\d{x} \nonumber\\
&= \sum\limits_{\ell=0}^{L} \frac{1}{n_\ell}
	\int_\physu
	\var{ \chlvl{\ell}{}Y_\ell }
	\d{x},
\label{eq:stochastic_error}
\end{align}
where $\physu$ represents the physical domain on which the control $u\in U$ is defined.
Clearly, the stochastic error can be decreased by taking more samples. 
Let $V_\ell \triangleq \int_\physu \var{ \chlvl{\ell}{}Y_\ell } \d{x}$. We assume the following:
\begin{assumption}
	The decay of $V_\ell$ satisfies $V_\ell \lesssim 2^{-\phi\ell}$ for some $\phi > 0$.
	\label{as:phi}
\end{assumption}
A bound of some prescribed value $\epsilon$ on the RMSE can be achieved by bounding the stochastic error term in (\ref{eq:MLMC_error}) by $\theta \epsilon^2$ and the bias term by $(1-\theta)\epsilon^2$ for some $\theta \in [0,1]$.

\subsection{Cost}
\label{sec:MLMC/cost}
\newcommand{\mlcost}{\mathcal{C}^\text{ML}_{\vec{n}}}
Denoting the cost of taking a sample of $Y_\ell$ as $\mathcal{C}_\ell$, the cost $\mlcost$ of the MLMC estimator can be expressed as
$\mlcost = \sum_{\ell = 0}^{L} n_\ell \mathcal{C}_\ell.$
Minimizing this cost such that (\ref{eq:stochastic_error}) is bounded by $\theta\epsilon^2$, i.e., such that $\sum_{\ell=0}^{L}n_\ell^{-1}V_\ell = \theta\epsilon^2$, and rounding upward, yields the optimal number of samples
\begin{equation}
n_\ell = \Bigg\lceil \frac{1}{\theta\epsilon^2} \sqrt{V_\ell\mathcal{C}_\ell^{-1}} \sum\limits_{i=0}^{L} \sqrt{V_i\mathcal{C}_i }\Bigg\rceil. \label{eq:MLMC_n}
\end{equation}
In order to obtain numerical values for $\vec{n}$, we need to estimate the values $V_\ell$. Note that $V_\ell \triangleq \int_D \var{ \chlvl{\ell}{}Y_\ell } \d{x} \approx \int_D \var{Y_\ell} \d{x}$. The integral can be approximated using an appropriate quadrature formula in $U_\ell$. The variance is estimated using some warm-up samples. For the finest levels, one can extrapolate using an estimation of the decay rate $\gamma$ in Assumption \ref{as:phi}. This is advantageous, since the number of samples that should be taken at those levels may be smaller than any reasonable number of warm-up samples.

Under Assumptions \ref{as:rhokappa} and \ref{as:phi}, the cost of the MLMC method can be bounded as presented in $\cite{cliffe2011, giles2015}$:
\begin{theorem}[Multilevel Monte Carlo cost]
	Suppose that there are positive constants $\rho, \phi, \kappa > 0$ such that $\rho > \frac{1}{2} \min (\phi, \kappa)$ and
	\begin{align*}
	\norm{\mean{\chlvl{\ell}{}Q_\ell - Q}} \lesssim 2^{-\rho\ell} &&
	V_\ell \lesssim 2^{-\phi\ell} &&
	\mathcal{C}_\ell \lesssim 2^{\kappa\ell}
	\end{align*}
	Then, for any $\epsilon < e^{-1}$, a value $L$ and a sequence $\vec{n} = \{n_\ell\}_{\ell = 0}^L$ exist such that the MSE
	\begin{equation*}
	\mean{\norm{\chlvl{L}{}\QMLMC - \mean{Q}}^2}  < \epsilon^2
	\end{equation*}
	and the cost 
	\begin{equation}
	\mlcost \lesssim
	\left\{\begin{array}{ll}
	\epsilon^{-2} &\text{if } \phi > \kappa \\
	\epsilon^{-2}(\log \epsilon)^2 &\text{if } \phi = \kappa \\
	\epsilon^{-2-(\kappa-\phi)/\rho} &\text{if } \phi < \kappa \\
	\end{array}\right.
	\label{eq:MLMC_cost}
	\end{equation}
	\label{theorem:MLMC}
\end{theorem}
\begin{remark}
	This theorem was proven in \cite{cliffe2011} for scalar valued quantities of interest. However, since all assumptions in Theorem \ref{theorem:MLMC} are scalar and of the same form as the assumptions in \cite{cliffe2011}, the proof is completely analogous to the one in \cite{cliffe2011}, which is also noted in \cite{cliffe2011, giles2015}. 
\end{remark}
\begin{remark}
	The value for $L$ in the theorem is such that the bias term (\ref{eq:MLMC_bias}) is small enough. In fact, as is immediately clear from \cite{cliffe2011}, the theorem holds iff $L$ is chosen dependent on $\epsilon$ in such a way that the bias follows $\epsilon$, i.e., iff $\norm{\mean{\chlvl{L}{}Q_L - Q}}^2 \simeq \epsilon^2$.
\end{remark}

\subsection{MLMC gradient estimator}
The gradient $\nabla\rJ_L(u_L) \in U_L$ for input $u_L \in U_L$ estimated using MLMC based on sample sets $(\Omega_0,\ldots,\Omega_L)$ containing $\vec{n} = (n_0,\ldots,n_L)$ samples respectively, can be written as
\newcommand{\Ql}[2]{\ensuremath{{Q}_{{#1}}(\chlvl{\Lmax}{#1}{u_L}, \omega) }}
\begin{equation} 
	\nabla\rJ_L(u_L) = \alpha u_L + \sum\limits_{\ell=0}^{L} \chlvl{\ell}{\Lmax}\frac{1}{n_\ell}\sum\limits_{\omega\in\Omega_\ell} \big(\Ql{\ell}{\ell} - \chlvl{\ell-1}{\ell}\Ql{\ell-1}{\ell}\big),
	\label{eq:grad_Q}
\end{equation}
where the dependence of $Q_\ell$ on an approximation of $u$ in $U_\ell$ has been made explicit. Let $\nabla \rJ_\ell^s(u_\ell) \triangleq \alpha u_\ell + n_s^{-1}\sum_{\omega \in \Omega_s} Q_\ell(u_\ell, \omega)$ 
denote the MC estimator based on $n_s$ samples $\Omega_s$ of $\alpha u_\ell + Q_\ell(u_\ell, \und)$ for some $u_\ell \in U_\ell$. 
Using the fact that the MLMC estimator is the identity for the deterministic $\alpha u_L$, (\ref{eq:grad_Q}) can be alternatively expressed as 
\begin{equation} 
	\nabla \rJ_L(u_L) = \chlvl{0}{\Lmax}\nabla\rJ_0^0(\chlvl{\Lmax}{0}{u_L}) + 
	\sum_{\ell = 1}^{L} \chlvl{\ell}{\Lmax} \big(
	\nabla\rJ_\ell^\ell(\chlvl{\Lmax}{\ell} {u_L}) - \chlvl{\ell-1}{\ell} \nabla\rJ_{\ell-1}^{\ell}(\chlvl{\Lmax}{\ell-1} {u_L})\big).
	\label{eq:grad_MLMC}
\end{equation}

Let $\rJ_L$ denote the cost functional estimated using MLMC with the same samples on the same levels as $\nabla \rJ_L$:
\begin{equation}
\rJ_L(u_L) = \rJ_0^0(\chlvl{\Lmax}{0}u_L) + 
\sum_{\ell = 1}^{L} 
\big(\rJ_\ell^\ell(\chlvl{\Lmax}{\ell} u_L) 
- \rJ_{\ell-1}^{\ell}(\chlvl{\Lmax}{\ell-1} u_L)\big),
\label{eq:cost_MLMC}
\end{equation}
where $\rJ_\ell^s: U_\ell \rightarrow \mathbb{R}$ denotes the Monte Carlo estimation of the cost functional in (\ref{eq:reduced_robust_problem}) using samples $\Omega_s$.
The notations are justified by the following theorem:
\begin{theorem}[exactness of the MLMC gradient]
	\label{theorem:MLMCexactness}
	The gradient $\nabla \rJ_L$ as given in (\ref{eq:grad_MLMC})
	is the exact gradient of $\rJ_L$ as given in (\ref{eq:cost_MLMC}), w.r.t. the inner product $\ip{\und}{\und}_L$.
\end{theorem}
\begin{proof}
	The previously defined $\nabla \rJ_\ell^s : U_\ell \rightarrow U_\ell$ is the gradient of $\rJ_\ell^s$, since the derivation in \S\ref{sec:optconds} holds for any probability space, and therefore also for a space $\Omega_s$ with a finite number of equally probable samples.
	Consider $u_L \in U_L$ given at the level $L$. The chain rule and (\ref{eq:chlvl_adjoint}) yield
	\begin{equation*}
	\diff{}{u_L} \rJ_\ell^k(\chlvl{\Lmax}{\ell}u_L)[h_L] = (\nabla \rJ_\ell^k(\chlvl{\Lmax}{\ell}u_L), \chlvl{\Lmax}{\ell}h_L)_\ell = (\chlvl{\ell}{\Lmax}\nabla \rJ_\ell^k(\chlvl{\Lmax}{\ell}u_L), h_L)_L,
	\label{eq:chain_gradient}
	\end{equation*}
	from which the theorem can easily be checked.
\end{proof}
In general, for a given accuracy $\epsilon$, there is a difference in the distribution $\vec{n}$ of the samples over the levels between the optimal MLMC estimator for the cost functional on one hand, and for its gradient on the other hand.
In this paper, the MLMC method is constructed for the gradient, since having the gradient at a known accuracy is most important for optimization purposes \cite{vanbarel2019robust}.

\section{MG/OPT}
\label{sec:MG/OPT}
\newcommand{\tJ}{\mathcal{J}}
MG/OPT is a multigrid approach to optimization problems \cite{nash2000, lewis2005}. It is inspired by and resembles the so-called FAS (full approximation storage) scheme applied to the optimality conditions \cite{borzi2009multigrid, briggs2000multigrid}. It requires approximate optimization problems at every MG/OPT level $k \in {0,\dots, K}$:
\begin{equation}
\min_{u_k \in U_k} \rJ_k(u_k).
\label{eq:opt_prob_k}
\end{equation}
An MG/OPT V-cycle improves a starting approximation $v_k$ of the solution $u_k$ that solves 
\begin{equation}
\min_{u_k \in U_k} \tJ_k(u_k) \triangleq \rJ_k(u_k) - \ip{\tau_k}{u_k}_k
\label{eq:opt_prob_k2}
\end{equation}
for a given correction term $\tau_k$, whose function it is to remedy the discrepancies between the approximate optimization problems (\ref{eq:opt_prob_k}). The V-cycle employs some convergent optimization algorithm for problem (\ref{eq:opt_prob_k2}) at level $k$, one iteration of which is denoted by $S_k$.
In multigrid terminology, $S_k$ fulfills the role of `smoother'.
The V-cycle is made explicit in Algorithm \ref{alg:MG/OPT} below. Initially it must be called as MG/OPT V-cycle($v_K$, $0$, $K$) and it calls itself recursively.
\begin{algorithm}[H]
	\caption{MG/OPT V-cycle($v_k$, $\tau_k$, $k$)}
	\label{alg:MG/OPT}
	\begin{algorithmic}[1]
		\If{$k=0$}
		\State \Return $v_{0}' \leftarrow 
		S_0^{\nu_0 + \mu_0}(v_0)$ (or exact minimizer $u_0$ of (\ref{eq:opt_prob_k2})) \Comment Trivial case
		\EndIf
		\State $v_{k,1} \leftarrow S_k^{\nu_k}(v_k)$ \Comment Presmoothing \label{line:presmoothing}
		\State $v_{k-1} \leftarrow \chlvl{k}{k-1} v_{k,1}$ \label{line:afterpresmoothing}
		\State $\tau_{k-1} \leftarrow \chlvl{k}{k-1}\tau_k + 
		\nabla \rJ_{k-1}(v_{k-1}) - \chlvl{k}{k-1}\nabla \rJ_k(v_{k,1})$ \label{line:correction_term}
		\State $v_{k-1}'\leftarrow$ MG/OPT V-cycle($v_{k-1}$, $\tau_{k-1}$, $k-1$)
		\Comment Update at the coarse level
		\State $d_{k-1}' \leftarrow v_{k-1}' - v_{k-1}$
		\State $d_k \leftarrow \chlvl{k-1}{k}d_{k-1}'$
		\State $v_{k,2} \leftarrow v_{k,1} + s d_k$ with $s\geq 0$, see \S\ref{sec:opt/linesearch}
		\State $v_{k}' \leftarrow S_k^{\mu_k}(v_{k,2})$ \Comment Postsmoothing
		\State \Return $v_{k}'$
	\end{algorithmic}
\end{algorithm}

\begin{remark}
	In the description of Algorithm \ref{alg:MG/OPT} above, all variables are named in such a way that they are assigned to only once, even across the algorithm's own recursive invocations. 
	This is the reason that the name of some variables is extended with the symbol $ '$. 
	Furthermore, each variable is known only by a single name; there is no aliasing. 
	In a real implementation, one would prefer to do the opposite by overwriting several variables to save memory and to prevent allocations of new memory.
\end{remark}

\begin{remark}
Line \ref{line:correction_term} can also be written as $\tau_{k-1} \leftarrow \nabla \rJ_{k-1}(v_{k-1}) - \chlvl{k}{k-1}\nabla \tJ_k(v_{k,1})$.
\end{remark}	

\subsection{Properties}
\begin{theorem} \label{theorem:MG/OPTgrad}
	Consider an MG/OPT level $k\geq 1$, then $\chlvl{k}{k-1}\nabla\tJ_k(v_{k,1}) = \nabla\tJ_{k-1}(v_{k-1})$, i.e., the gradient of the coarse grid problem in its initial point matches the restricted fine grid gradient at that point \cite{lewis2005}.
\end{theorem}
\begin{proof}
	The statement follows from
	\begin{align*}
		\nabla\tJ_{k-1}(v_{k-1}) &= \nabla\rJ_{k-1}(v_{k-1}) - \tau_{k-1} \\ 
		& = \nabla\rJ_{k-1}(v_{k-1}) - \chlvl{k}{k-1}\tau_k -
			\nabla \rJ_{k-1}(v_{k-1}) + \chlvl{k}{k-1}\nabla \rJ_k(v_{k,1}) \\
		& = \chlvl{k}{k-1}\nabla \rJ_k(v_{k,1}) - \chlvl{k}{k-1}\tau_k = \chlvl{k}{k-1}\nabla\tJ_k(v_{k,1}). \qedhere
	\end{align*} 
\end{proof}
\begin{corollary}
	Assuming $S_k$ maps the optimal point $u_k$ to itself, then the MG/OPT V-cycle does so for $u_K$. Indeed, if $v_K = u_K$, the gradients $\nabla \tJ_k(v_k)$ and corrections $d_k$ are zero at all levels $k$.
\end{corollary}
It is also instructive to state the following theorem from \cite{borzi2009multigrid}.
\begin{theorem} \label{theorem:descent}
	Assume descent happens at MG/OPT level $k-1$, then the returned search direction $d_k$ is a descent direction at MG/OPT level $k$, i.e,
	$$\tJ_{k-1}(v_{k-1}') < \tJ_{k-1}(v_{k-1}) \implies \ip{\nabla\tJ_k(v_{k,1})}{d_k}_k < 0$$
	if $\hess{\tJ_{k-1}}$ is positive definite on a line connecting $v_{k-1}$ and $v_{k-1}'$.
\end{theorem}
\begin{proof}
	For some $v,z \in U_{k-1}$, consider the expansion
	\begin{equation*}
		\tJ_{k-1}(v+z) = \tJ_{k-1}(v) + \ip{\nabla\tJ_{k-1}(v)}{z}_{k-1} 
		+ \frac{1}{2}\int_{0}^{1}\ip{\hess{\tJ_{k-1}(v+tz)}[z]}{z}_{k-1}\d{t}.
	\end{equation*}
	Using the assumption on the Hessian and choosing $v = v_{k-1}$ and $z = v_{k-1}'-v_{k-1}$,
	\begin{equation*}
		\ip{\nabla \tJ_{k-1}(v_{k-1})}{v_{k-1}'-v_{k-1}}_{k-1}
		\leq 
		\tJ_{k-1}(v_{k-1}') - \tJ_{k-1}(v_{k-1}) < 0.
	\end{equation*}
	Now, using Theorem \ref{theorem:MG/OPTgrad} and (\ref{eq:chlvl_adjoint}),
	\begin{align*}
		\ip{\nabla\tJ_k(v_{k,1})}{d_k}_k &= \ip{\nabla\tJ_k(v_{k,1})}{\chlvl{k-1}{k}(v_{k-1}'-v_{k-1})}_k \\
		&= \ip{\chlvl{k}{k-1}\nabla\tJ_k(v_{k,1})}{v_{k-1}'-v_{k-1}}_{k-1} \\
		&= \ip{\nabla\tJ_{k-1}(v_{k-1})}{v_{k-1}'-v_{k-1}}_{k-1} < 0. \qedhere
	\end{align*}
\end{proof}

\subsection{Line search}
\label{sec:opt/linesearch}
Some of the benefits of MG/OPT depend crucially on the use of a line search along the coarse grid correction direction. This ensures global convergence of the method, meaning that the method converges to a local minimizer regardless of the starting guess. MG/OPT without line search might diverge in some cases \cite{nash2000}. If the problem is quadratic, $s=1$ is a good choice. This is also the case if one assumes that the problem can be well approximated by a quadratic cost functional in a neighborhood of the optimal point, with the control iterates in that neighborhood. 

Alternatively, in \cite{borzi2005convergence}, Borz\`i presents a more sophisticated a priori choice of $s$ under the assumption that the Hessian of $\tJ_k$ (equal of course to the Hessian of $\rJ_k$) is elliptic and Lipschitz continuous with known Lipschitz constant $L_k>0$. In that case, on MG/OPT level $k$, one should consider the choice
\newcommand{\vk}{\ensuremath{v_{k}}}
\begin{equation}
s = \min\Bigg\{2,\frac{\ip{-\nabla \tJ_k(\vk)}{d_k}_k}
{\ip{\hess{\tJ_k(\vk)}[d_k]}{d_k}_k + L_k\norm{d_k}^3_k} \Bigg\}.
\end{equation}

In the experiments in \S\ref{sec:numresults}, however, we use backtracking until we have descent, starting with $s=1$. If we have descent, the cost function and gradient evaluation performed to come to that conclusion can be reused as the first cost function gradient pair in the postsmoothing that comes after. Therefore, if the choice $s=1$ already yields descent, no additional work is done. This is more and more likely the further along the optimization one is, if one assumes the problem to be quadratic near the optimal point. This backtracking is thus likely to happen only in the early stages of the optimization, where gradient evaluations are still cheap, as will be discussed later. The costs incurred should therefore remain minimal. Either way, if the Hessian condition in Theorem \ref{theorem:descent} holds, we are guaranteed to find a suitable $s$ yielding descent.

\subsection{Specification of MG/OPT levels}
Assume that the gradient for the finest MG/OPT problem is calculated using MLMC for some requested RMSE $\epsilon$, resulting in a number of samples $(n_{0,K},\ldots,n_{K,K})$ that have to be taken at the levels $0,\ldots,K$. We denote the actual sample sets as $(\Omega_{0,K},\ldots,\Omega_{K,K})$. 
Gradients for the finest problem are returned on level $K$.
It then remains to specify the coarser optimization problems. We propose to let the coarser problems correspond to problems in which the expected values are also approximated using MLMC, but retaining only the coarser levels and a fraction $q$ of the samples on those levels.
More specifically, for $k \in \{0,\dots,K\}$, the MLMC estimation at MG/OPT level $k$ consists of $(n_{0,k},\ldots,n_{k,k}) = q^{K-k}(n_{1,K}, \dots, n_{k,K})$ samples $(\Omega_{0,k},\ldots,\Omega_{k,k})$ on levels $0,\ldots,k$. This is illustrated in Figure \ref{fig:class_detailed} for $K=3$. Quantities such as the gradient are returned on level $k$. 
Writing down the definition explicitly yields, in analogy with (\ref{eq:cost_MLMC}) and (\ref{eq:grad_MLMC})
\renewcommand{\Ql}[2]{\ensuremath{{Q}_{{#1}}(\chlvl{k}{#1}{u_k}, \omega) }}
\begin{align}
\rJ_k(u_k) &= \rJ_0^{0,k}(\chlvl{k}{0}u_k) + 
\sum_{\ell = 1}^{k} 
\big(\rJ_\ell^{\ell,k}(\chlvl{k}{\ell} u_k) 
- \rJ_{\ell-1}^{\ell,k}(\chlvl{k}{\ell-1} u_k)\big)
\label{eq:costk_MLMC}
\\
\nabla \rJ_k(u_k) &= \chlvl{0}{k}\nabla\rJ_0^{0,k}(\chlvl{k}{0}{u_k}) + 
\sum_{\ell = 1}^{k} \chlvl{\ell}{k} \big(
\nabla\rJ_\ell^{\ell,k}(\chlvl{k}{\ell} {u_k}) - \chlvl{\ell-1}{\ell} \nabla\rJ_{\ell-1}^{\ell,k}(\chlvl{k}{\ell-1} {u_k})\big),
\label{eq:gradk_MLMC}
\end{align}
where $\nabla \rJ_\ell^{s,k}(u_\ell) \triangleq \alpha u_\ell + n_{s,k}^{-1}\sum_{\omega \in \Omega_{s,k}} Q_\ell(\omega, u_\ell)$ is the gradient of $\rJ^{s,k}_\ell$.

Note that for any $\ell\leq k$, $\Omega_{\ell,k}$ needs not be a subset of $\Omega_{\ell,k+1}$, but numerical experiments show no disadvantage if it is. Furthermore, if $\Omega_{\ell,k} \subset \Omega_{\ell,k+1}$, the MG/OPT levels are nested, removing the need to calculate $\nabla \rJ_{k-1}(\chlvl{k}{k-1}v_k)$ in Line \ref{line:correction_term}, since it can be obtained from $\nabla \rJ_k(v_k)$. However, it is important that $\forall \ell \neq \ell': \Omega_{\ell,k} \cap \Omega_{\ell',k} = \phi$. On a single MG/OPT level $k$, reusing samples on different MLMC levels yields estimators $\YMC_{\ell,n_\ell}$ in (\ref{eq:MC_def1Y}) that are no longer independent, causing the theory of the previous chapter to break down.

\begin{figure}[tb]
	\begin{subfigure}[t]{.49\textwidth}
		\centering
		\caption{Number of samples $n_{\ell,k}$ on MLMC level $\ell$ for MG/OPT level $k$. $q = 2^{-2\rho}$ with $\rho$ the convergence order in (\ref{eq:disc_error}).}
		\begin{tikzpicture}[scale=0.8]
\def\h{0.6};
\def\w{1.5};

\draw[thick,->] (0,0) -- ({4.0*\w},0); 
\draw[thick,->] (0,0) -- (0,{4.0*\h}); 

\node at ({2.0*\w}, {-0.9}) {MG/OPT level $k$};
\node[rotate=90] at (-1.0, {1.5*\h}) {Discr. level $\ell$};

\foreach \x in {0,1,2,3}
\draw ({\w*\x+\w*0.5},0.1) -- ({\w*\x+\w*0.5},-0.1) node[anchor=north] {$\x$};
\foreach \y in {0,1,2,3}
\draw (0.1,{(\y+0.5)*\h}) -- (-0.1,{(\y+0.5)*\h}) node[anchor=east] {$\y$};

\node at ({\w*0.5},{0.5*\h}) {$q^3n_0$};

\node at ({\w*1.5},{1.5*\h}) {$q^2n_1$};
\node at ({\w*1.5},{0.5*\h}) {$q^2n_0$};

\node at ({\w*2.5},{2.5*\h}) {$qn_2$};
\node at ({\w*2.5},{1.5*\h}) {$qn_1$};
\node at ({\w*2.5},{0.5*\h}) {$qn_0$};

\node at ({\w*3.5},{3.5*\h}) {$n_3$};
\node at ({\w*3.5},{2.5*\h}) {$n_2$};
\node at ({\w*3.5},{1.5*\h}) {$n_1$};
\node at ({\w*3.5},{0.5*\h}) {$n_0$};

\end{tikzpicture}
		\label{fig:class_detailed}
	\end{subfigure}
	\quad
	\begin{subfigure}[t]{.49\textwidth}
		\centering
		\caption{Number of smoothing (optimization) steps $\nu_k$ and $\mu_k$ in Algorithm \ref{alg:MG/OPT} at each MG/OPT level $k$.}
		\begin{tikzpicture}[scale=0.8]
\def\h{0.6};

\draw[thick] (0,4*\h) -- (2,0) -- (4,{4.0*\h}); 

\def\presmoothsteps{{"0","1","2","4+4"}}
\def\postsmoothsteps{{"","2","1","1"}}

\foreach \x in {0,1,2,3}
\fill (\x*0.666666,{4*\h-\x*\h*1.333333}) circle (3pt) node[anchor=west] 
{\pgfmathparse{\presmoothsteps[\x]}$\pgfmathresult$}
;

\foreach \x in {1,2,3}
\fill ({2+\x*0.666666},{\x*\h*1.333333}) circle (3pt) node[anchor=west] 
{\pgfmathparse{\postsmoothsteps[\x]}$\pgfmathresult$}
;

\foreach \y in {0,1,2,3}
\draw[dashed] ({(3-\y)*2/3},{\y*\h*1.333333}) -- (-0.75,{\y*\h*1.333333}) node[anchor=east] {$\y$};
\node[rotate=90] at (-1.75, {2*\h}) {MG/OPT level $k$};

\end{tikzpicture}
		\label{fig:Vcycle}
	\end{subfigure}
	\caption{Illustration of the V-cycle structure for 4 MG/OPT levels, i.e., for $K=3$.}
\end{figure}
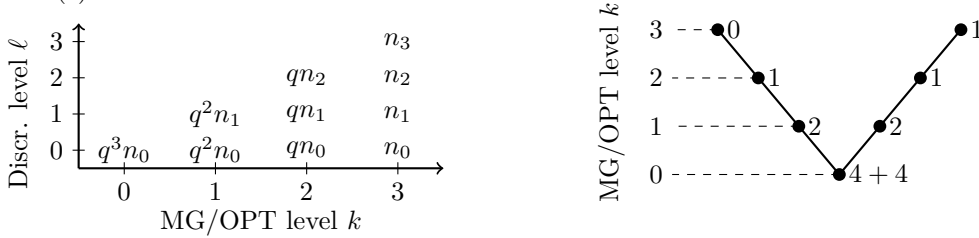

The number $q$ is derived from the order of convergence of the discretization scheme (\ref{eq:disc_error}). The discretization error on level $\ell-1$ can be expected to be $2^{-\rho(\ell-1)}/2^{-\rho\ell} = 2^\rho$ times larger than on level $\ell$. If one allows the stochastic error $\epsilon_{\text{stoch}} = \mathcal{O}(n^{-1/2})$ to also be that much larger, the number of samples can be reduced by a factor $q^{-1}=2^{2\rho}$.
Usually, $1\leq\rho\leq2$, implying $1/4\leq q\leq1/16$.
\begin{remark}
	Alternatively, one might simply request a RMSE that is $2^\rho$ times smaller for a coarser problem, following the rate of convergence of the discretization scheme. The resulting $n_{\ell,k}$ would be very similar. 
\end{remark}

\subsection{Cost}
\newcommand{\rf}{2}
Let $\mathcal{C}_{\text{g},k}$ denote the computational cost of a gradient evaluation at MG/OPT level $k$. Then 
$$\mathcal{C}_{\text{g},k} = q^{K-k}\sum_{\ell=0}^k n_k\mathcal{C}_\ell.$$
Clearly $\mathcal{C}_{\text{g},k-1} < q\,\mathcal{C}_{\text{g},k}$; a gradient evaluation at MG/OPT level $k-1$ is more than $q$ times cheaper than one at level $k$.
This leaves computational room to do twice as many smoothing steps at level $k-1$ than at level $k$. A possible smoothing schedule is illustrated in Figure \ref{fig:Vcycle}.
Denote the cost of a single MG/OPT V-cycle as $\mathcal{C}_\text{V}$ and the cost of a smoothing step at level $k$ as $\mathcal{C}_{\text{s},k}$. Assuming the smoother requires $c$ gradient evaluations, $\mathcal{C}_{\text{s},k} = c\,\mathcal{C}_{\text{g},k}$. Requiring $q<\frac{1}{2}$, we then have
\begin{align}
\mathcal{C}_\text{V} &= \sum_{k=0}^{K} \rf^{K-k}\mathcal{C}_{\text{s},k} + \mathcal{C}_{g,k} = \sum_{k=0}^{K} (c\rf^{K-k}+1)\mathcal{C}_{\text{g},k} \nonumber \\
&< \sum_{k=0}^{K} (c(\rf q)^{K-k} + q^{K-k})\mathcal{C}_{\text{g},K} < \left(\frac{c}{1-\rf q} + \frac{1}{1-q}\right)\mathcal{C}_{\text{g},K}.
\label{eq:costV}
\end{align}
The gradient cost not related to the smoother and that is incurred at each level comes from the need\footnote{It is possible that the presmoother can return the gradient in its final point at no additional cost, which allows this to be free.} to calculate $\tau_{k-1}$ at Line \ref{line:correction_term}. As explained in \S\ref{sec:opt/linesearch}, the line search usually does not come with additional gradient evaluations.
If one does not do any presmoothing, then $\tau_k$ can be simultaneously derived for all $k$ with a single gradient simulation at level $K$. This can be understood from, e.g., (\ref{eq:gradk_MLMC}). For a given $u_k \in U_k$, assuming nested MG/OPT levels, the samples of $Q_\ell$, $\ell = 0,\ldots,k-1$ required to calculate $\nabla \rJ_{k-1}(\chlvl{k}{k-1}u)$ are a subset of the samples required to calculate $\nabla \rJ_k(u_k)$. This leads to the modified cost
\begin{align}
\mathcal{C}_\text{V}' &= \mathcal{C}_{g,K} + \sum_{k=0}^{K} \rf^{K-k}\mathcal{C}_{\text{s},k} = \mathcal{C}_{g,K} + c\sum_{k=0}^{K} \rf^{K-k}\mathcal{C}_{\text{g},k} \nonumber \\
&< \mathcal{C}_{g,K} + c\sum_{k=0}^{K} (\rf q)^{K-k}\mathcal{C}_{\text{g},K} < \left(\frac{c}{1-\rf q}+1\right)\mathcal{C}_{\text{g},K}.
\label{eq:costValt}
\end{align}
Since $q<\frac{1}{2}$, the cost savings can be at most $\frac{1}{2}\mathcal{C}_{g,K}$.
The cost for $\rho=1$ and $\rho=2$ can be found below, for a general value of $c$ and for $c=2$, as will be the case in the experiments at the end of the paper.
\begin{equation}
	\begin{array}{ c c c c c c }
		\rho & q & \mathcal{C}_\text{V} & \mathcal{C}_\text{V}' & \mathcal{C}_\text{V}, c=2 & \mathcal{C}_\text{V}', c=2 \\ \hline 
		1 & 1/4 & (2c+\frac{4}{3})\mathcal{C}_{\text{g},K} & (2c+1)\mathcal{C}_{\text{g},K} & 5.33\mathcal{C}_{\text{g},K} & 5.00\mathcal{C}_{\text{g},K} \\  
		2 & 1/16 & (\frac{8}{7}c+\frac{16}{15})\mathcal{C}_{\text{g},K} & (\frac{8}{7}c+1)\mathcal{C}_{\text{g},K} & 3.35\mathcal{C}_{\text{g},K} & 3.29\mathcal{C}_{\text{g},K}
	\end{array}
\end{equation}

\subsection{Nonlinear conjugate gradient `smoother'}
\label{sec:opt/smoother}
\newcommand{\itr}[1]{\ensuremath{^{(#1)}}}
In theory, any convergent optimization algorithm can be used as the smoother. In the experiments in this paper, the nonlinear conjugate gradient (NCG) method is employed, both for pre- and postsmoothing. The smoother calls in Algorithm \ref{alg:MG/OPT} of the form $S_k^J(v\itr{0})$, with $v\itr{0} \in U_k$, perform $J$ iterations at discretization level $k$.
In an NCG iteration $j\leq J$, a search direction
$d\itr{j} \in U_k$ is obtained as the linear combination of the steepest descent direction and the previously used direction
$$d\itr{j} = -g\itr{j} + \beta\itr{j} d\itr{j-1},$$
where $g\itr{j} = \nabla \tJ_k(v\itr{j}) \in U_k$ denotes the gradient in iteration $j$. Initially we use $d\itr{0} = -g\itr{0}$. 
Several formulas for $\beta\itr{j}$ exist \cite{Borzi2012}, suited for different problems and having different convergence properties. We use the Dai-Yuan formula
\begin{equation} 
	\label{eq:Dai-Yuan}
	\beta\itr{j} = \frac{\|g\itr{j}\|^2_k}{\ip{d\itr{j-1}}{g\itr{j} - g\itr{j-1}}_k},
\end{equation}
which may offer certain globalization advantages \cite{Dai1999}.
The current iterate is then to be updated as 
$v\itr{j+1} = v\itr{j} + s\itr{j}d\itr{j}$
with $s\itr{j}$ some acceptable step size.
For a quadratic problem, a line search can be performed exactly using one additional gradient evaluation along the search direction, say in the point $v\itr{j} + d\itr{j}$. Furthermore, due to its linearity in that case, the gradient in the new point can be obtained for free as 
$$
\nabla \tJ_k(v\itr{j+1}) = \nabla \tJ_k(v\itr{j} + s\itr{j}d\itr{j}) = (1-s\itr{j})\nabla \tJ_k(v\itr{j}) + s\itr{j}\nabla \tJ_k(v\itr{j} + d\itr{j}).
$$
Each iteration therefore really requires just one gradient evaluation. In fact, for quadratic optimization problems, this whole procedure is equivalent to applying the standard conjugate gradient method on the symmetric linear system of optimality conditions \cite{nazareth2009}. This is the case for (most) common formulas for $\beta\itr{j}$, including (\ref{eq:Dai-Yuan}).
If the problem is not quadratic, the gradient function is nonlinear. An acceptable step size $s\itr{j}$ is then sought in a problem specific manner. Two possibilities are as follows:
\begin{itemize}
	\item Approximate the cost function along the search direction with an interpolating parabola using a second gradient evaluation and perform the procedure as described previously. Note that the interpolating parabola might be concave. A sufficient descent condition is thus needed. Furthermore, the gradient in the resulting point can no longer be obtained for free.
	\item Start with a well chosen step size, and do backtracking until Armijo's sufficient descent condition \cite{armijo1966} is satisfied.
\end{itemize}
The details and further implementational aspects of the smoother depend on the test case, and are elaborated in \S \ref{sec:numresults}.

\section{Qualitative comparison with existing methods}
\label{sec:qualcomp}
This section aims to provide a qualitative comparison of the method described previously with existing methods that have some resemblance to it. To that end, some methods are sketched in Figure \ref{fig:classes}. In the interest of keeping this section sufficiently concise, some slight simplifications are made in the descriptions of those methods.
\begin{figure}[h]
	\raggedright
	\begin{subfigure}[t]{.45\textwidth}
		\centering
		\caption{MLMC on gradient \cite{vanbarel2019robust}.}
		\begin{tikzpicture}[scale=0.8]
\def\h{0.6};

\draw[thick,->] (0,0) -- (4.0,0); 
\draw[thick,->] (0,0) -- (0,{4.0*\h}); 

\node at (2.0, {-0.9}) {MG/OPT level $k$};
\node[rotate=90] at (-1.0, {1.5*\h}) {Discr. level $\ell$};

\foreach \x in {1,2,3,4}
\draw ({\x-0.5},0.1) -- ({\x-0.5},-0.1) node[anchor=north] {$\x$};
\foreach \y in {1,2,3,4}
\draw (0.1,{(\y-0.5)*\h}) -- (-0.1,{(\y-0.5)*\h}) node[anchor=east] {$\y$};

\node at (3.5,{0.5*\h}) {$4$};
\node at (3.5,{1.5*\h}) {$3$};
\node at (3.5,{2.5*\h}) {$2$};
\node at (3.5,{3.5*\h}) {$1$};
\end{tikzpicture}
		\label{fig:class1}
	\end{subfigure}
	\quad
	\begin{subfigure}[t]{.45\textwidth}
		\centering
		\caption{MG/OPT + MLMC, [This paper].}
		\begin{tikzpicture}[scale=0.8]
\def\h{0.6};

\draw[thick,->] (0,0) -- (4.0,0); 
\draw[thick,->] (0,0) -- (0,{4.0*\h}); 

\node at (2.0, {-0.9}) {MG/OPT level $k$};
\node[rotate=90] at (-1.0, {1.5*\h}) {Discr. level $\ell$};

\foreach \x in {1,2,3,4}
\draw ({\x-0.5},0.1) -- ({\x-0.5},-0.1) node[anchor=north] {$\x$};
\foreach \y in {1,2,3,4}
\draw (0.1,{(\y-0.5)*\h}) -- (-0.1,{(\y-0.5)*\h}) node[anchor=east] {$\y$};

\node at (0.5,{0.5*\h}) {$1$};

\node at (1.5,{1.5*\h}) {$1$};
\node at (1.5,{0.5*\h}) {$2$};

\node at (2.5,{2.5*\h}) {$1$};
\node at (2.5,{1.5*\h}) {$2$};
\node at (2.5,{0.5*\h}) {$3$};

\node at (3.5,{3.5*\h}) {$1$};
\node at (3.5,{2.5*\h}) {$2$};
\node at (3.5,{1.5*\h}) {$3$};
\node at (3.5,{0.5*\h}) {$4$};

\end{tikzpicture}
		\label{fig:class2}
	\end{subfigure}\\
	\vspace{0.2cm}
	\begin{subfigure}[t]{.45\textwidth}
		\centering
		\caption{MG/OPT in stochastic space \cite{Kouri2014}. \phantom{phantom phantom phantom phantom phantom phantom phantom phantom}}
		\begin{tikzpicture}[scale=0.8]
\def\h{0.6};

\draw[thick,->] (0,0) -- (4.0,0); 
\draw[thick,->] (0,0) -- (0,{4.0*\h}); 

\node at (2.0, {-0.9}) {MG/OPT level $k$};
\node[rotate=90] at (-1.0, {1.5*\h}) {Discr. level $\ell$};

\foreach \x in {1,2,3,4}
\draw ({\x-0.5},0.1) -- ({\x-0.5},-0.1) node[anchor=north] {$\x$};
\foreach \y in {1,2,3,4}
\draw (0.1,{(\y-0.5)*\h}) -- (-0.1,{(\y-0.5)*\h}) node[anchor=east] {$\y$};

\node at (0.5,{3.5*\h}) {$1$};
\node at (1.5,{3.5*\h}) {$2$};
\node at (2.5,{3.5*\h}) {$3$};
\node at (3.5,{3.5*\h}) {$4$};

\end{tikzpicture}
		\label{fig:class3}
	\end{subfigure}
	\quad
	\begin{subfigure}[t]{.45\textwidth}
		\centering
		\caption{Multigrid on equations after stochastic discretization, e.g., \cite{Rosseel2012}.}
		\begin{tikzpicture}[scale=0.8]
\def\h{0.6};

\draw[thick,->] (0,0) -- (4.0,0); 
\draw[thick,->] (0,0) -- (0,{4.0*\h}); 

\node at (2.0, {-0.9}) {MG/OPT level $k$};
\node[rotate=90] at (-1.0, {1.5*\h}) {Discr. level $\ell$};

\foreach \x in {1,2,3,4}
\draw ({\x-0.5},0.1) -- ({\x-0.5},-0.1) node[anchor=north] {$\x$};
\foreach \y in {1,2,3,4}
\draw (0.1,{(\y-0.5)*\h}) -- (-0.1,{(\y-0.5)*\h}) node[anchor=east] {$\y$};

\node at (3.5,{3.5*\h}) {$4$};
\node at (2.5,{2.5*\h}) {$4$};
\node at (1.5,{1.5*\h}) {$4$};
\node at (0.5,{0.5*\h}) {$4$};

\end{tikzpicture}
		\label{fig:class4}
	\end{subfigure}
	\caption{Qualitative comparison between different approaches. The figures give, as a function of MG/OPT level $k$ (or the equivalent) and PDE discretization level $\ell$, an indication of the number of samples taken. A number $1$ indicates that only very few (e.g., $10^1$) samples are taken. A number $4$ indicates that many (e.g., $10^4$) samples are taken.}
	\label{fig:classes}
\end{figure}
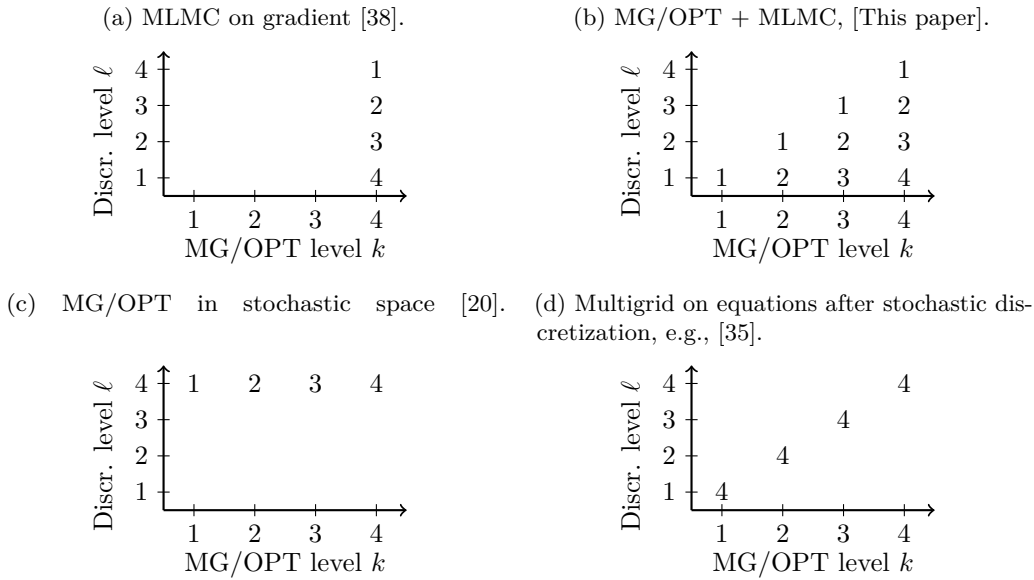

The first figure, Figure \ref{fig:class1}, shows a method that we previously developed in \cite{vanbarel2019robust}. It estimates the cost function and gradient using MLMC, but it performs the actual optimization steps on the finest level only. This exploits structure in the stochastic space, and was shown to reduce the computational cost by some orders of magnitude if compared to the naive single level MC method. Note that the latter method, if displayed in the format of Figure \ref{fig:classes}, would consist of merely a single `$4$' in the rightmost corner. However, the optimization itself might converge slowly, depending on the PDE constraint and the regularization of the cost function, regardless of the stochasticity of the problem.

Next, consider Figure \ref{fig:class3}, depicting a method by Kouri \cite{Kouri2014}. This method uses MG/OPT in a somewhat different way. Finer optimization problems correspond to estimations of the statistical quantities using more sample points. It must be noted that the method uses sparse grids instead of MC points to integrate over the stochastic domain. These sparse grids are nested to reduce the cost even further. There is only one discretization level for the space and time domains. This method, like the previous one, exploits structure in the stochastic space. Our suggestion of taking fewer samples on coarser MG/OPT levels was inspired by this paper. It was shown there that such a strategy is computationally advantageous. However, sooner or later, an optimization step on the fine MG/OPT level must happen. This entails taking many samples on the fine (and only) discretization level, shown as the `$4$' in the upper right corner of Figure \ref{fig:class3}. This is a problem that occurs also for the next method, represented in Figure \ref{fig:class4}. 

Figure \ref{fig:class4} shows the qualitative structure of a set of methods in which the number of realizations to estimate all statistical quantities is fixed first, essentially yielding a deterministic problem. 
Some multigrid method is then used on the resulting system of equations, containing all sampled realizations. Some methods use multigrid on the optimality conditions directly \cite{Rosseel2012, borzi2009multigrid}. Other methods might perform MG/OPT on the resulting deterministic optimization problem \cite{borzi2009multigrid}. The label for the horizontal axis in Figure \ref{fig:class4} may not be accurate for all methods; the label `Multigrid level $k$' might be better suited for some. These methods are often the first choice of people with a background in multigrid. Dealing with the uncertainties at the start leads to methods that essentially use the same number of samples on all discretization levels. In light of the discussions in \S\ref{sec:MLMC/cost}, this might be suboptimal. However, these methods can work well if the uncertainties are small, or low-dimensional, since then this fixed number of realizations might be small. 
We also note that, if the equations for all samples are solved simultaneously, as is the case in, e.g., \cite{Rosseel2012}, the memory requirement scales linearly with the number of samples. These methods do not really exploit structure in the stochastic space, but do exploit the PDE structure over the physical domain. The use of multigrid for dealing with PDEs is well established, and the speed at which the optimization problem itself is solved can thus be expected to be much improved.

Finally, we have the method described in this paper represented by Figure \ref{fig:class2}. The method attempts to combine the positive aspects of the previous methods. The structure in the stochastic domain is exploited by the use of MLMC, and the reduction of the number of samples on coarser MG/OPT levels. The structure inherent in the physical domains of the PDE is exploited by considering coarser discretization levels for coarser MG/OPT levels. It avoids the `$4$' in the upper right corner; at no point are many samples taken at the fine discretization level. 

\section{Robust optimization}
\label{sec:robopt}
Let $\tau$ denote the tolerance on the gradient norm. A practical algorithm to solve the optimization problem, with details on how to choose the gradient RMSE tolerance and a rudimentary stopping condition, is presented in Algorithm \ref{alg:opt}. The superscripts 
can be removed without changing the algorithm's functionality. Their only purpose is to allow a clearer description in what follows.
\newcommand{\cl}{r}
\newcommand{\gstart}{g^{(i)}_0}
\newcommand{\gend}{g^{(i)}}
\alglanguage{pseudocode}
\begin{algorithm}
	\caption{Robust optimization}
	\label{alg:opt}
	\begin{algorithmic}[1]
		\State input $\tau$, $\cl$, $\epsilon^{(1)}$, $i_\text{max}$, $v_K^{(0)}$, $\nabla \rJ_{\und}(\und)$
		\For {$i = 1,\ldots, i_\text{max}$}
		\State $(v_K^{(i)}, \norm{\gend}, \norm{\gstart}) \leftarrow $ MG/OPT V-cycle($v_K^{(i-1)}$, $0$, $K$)
		\Comment Using RMSE $\epsilon^{(i)}$
		\If{$\norm{\gend} \leq \tau$} \Comment Test convergence using V-cycle's samples
		\If{$\norm{\nabla \rJ_{K}(v_K^{(i)})} \leq \tau$} \Comment Using new random samples, RMSE $r\tau$ \label{line:convtest}
		\State \Return $v_K^{(i)}$
		\EndIf
		\EndIf
		\State $\eta^{(i)} \leftarrow \min\{1/2, \norm{\gend}/\norm{\gstart}\}$ \Comment Estimate the performance of the next V-cycle
		\State $\epsilon^{(i+1)} \leftarrow \max\{\cl\tau, \cl\eta^{(i)}\norm{\gend}\}$ \Comment Determine the next RMSE tolerance for the gradient \label{line:eps}
		\EndFor
	\end{algorithmic}
\end{algorithm}

Algorithm \ref{alg:opt} consists of a number (at most $i_\text{max}$) of MG/OPT V-cycles, which are described in Algorithm \ref{alg:MG/OPT} in \S\ref{sec:MG/OPT}.
At the start of each V-cycle, the sample sets, denoted earlier in \S\ref{sec:MG/OPT} by $\Omega_{k,\ell}$, are to be determined. During a V-cycle, the sample sets remain constant. 
A RMSE tolerance $\epsilon = \epsilon^{(i)}$ is used to determine $\vec{n}$ using (\ref{eq:MLMC_n}). 
A random seed is chosen, which, together with $\vec{n}$, then fully determines the sample sets $\Omega_{k,\ell}$. Recall that (\ref{eq:MLMC_n}) only ensures that the stochastic error is small enough. However, the RMSE (\ref{sec:MLMC/MSE}) also contains a bias term. For now, we don't worry about this, and assume that $K$ is such that the bias is small enough; see also Assumption \ref{as:K} below. 

The algorithm needs to be able to estimate how well the V-cycle did, i.e., how effective it was in reducing the gradient norm. In Algorithm \ref{alg:opt}, the call to MG/OPT V-cycle at Line $3$ therefore returns some additional information that was, for simplicity reasons, not mentioned in Algorithm \ref{alg:MG/OPT}.
The outputs $\smash{\norm{\gstart}}$ and $\smash{\norm{\gend}}$ provide the norm of the gradient at, respectively, the start and the end of the V-cycle, evaluated at the finest MG/OPT level using the sample sets employed in that V-cycle.\footnote{The first and last smoother call in the V-cycle may be able to return these gradients for free. Depending on the smoother, the final gradient might not be cheaply available. In that case, one may return the norm of the last available gradient instead, or produce an estimation in some other cheap way.}
The observed convergence rate of the V-cycle is then $\smash{\norm{\gend}/\norm{\gstart}}$. In Line 9, we let $\eta^{(i)}$ be this rate, unless it is worse than $1/2$, in which case we let $\eta^{(i)}$ be $1/2$. 
The purpose of $\eta^{(i)}$ is to serve as an indication of the convergence rate of V-cycle $i+1$. We ensure $\eta^{(i)} \leq 1/2$ to not be too pessimistic about the next V-cycle, even if the current V-cycle $i$ performed poorly. 

For optimization, a gradient is useful only if its accuracy is at least proportional to its norm. A smaller gradient is more easily distorted by a given amount of noise than a large gradient. Furthermore, the set of samples should be large enough such that the optimizer based on those samples would still work well for new random samples. Therefore, Algorithm \ref{alg:opt} attempts to have the RMSE satisfy
\begin{equation}
\epsilon^{(i)} \leq \cl\norm{g^{(i)}},
\label{eq:relRMSE}
\end{equation}
where the constant $\cl$
determines how large the relative RMSE $\epsilon^{(i)} / \norm{g^{(i)}}$ is allowed to be. In the experiments at the end, we set $r=1/2$. At the same time, $\epsilon^{(i)}$ cannot be too small, since then an unnecessarily large number of samples will be taken; see (\ref{eq:MLMC_n}). These ideas are behind Line $10$, where the RMSE $\epsilon^{(i+1)}$ for V-cycle $i+1$ is determined. 
If (\ref{eq:relRMSE}) holds, $\norm{g^{(i)}}$ can be expected to be a good indication for the actual gradient norm, i.e., the gradient norm as it would be if the full stochastic space $\Omega$ were considered \cite{vanbarel2019robust}.
Assuming $\eta^{(i)}$ is a good indicator for the performance of V-cycle $i+1$, we can expect
$\norm{g^{(i+1)}} \approx \eta \norm{g^{(i)}}$, and thus we need $\epsilon^{(i+1)} \leq \cl\norm{g^{(i+1)}} \approx \cl\eta \norm{g^{(i)}}$. Having $\eta^{(i)}\leq 1/2$ (see Line $9$), ensures that $\epsilon^{(i+1)}$ is sufficiently small to allow for some meaningful progress in the next V-cycle, even if the previous V-cycle performed poorly. Finally, since we are not interested in a gradient norm even smaller than the tolerance $\tau$, $r\tau$ is a lower bound on our gradient's RMSE requirement.
This gradual decrease of $\epsilon$, in tandem with the decrease of $\norm{g^{(i)}}$, ensures that
\begin{equation}
\epsilon^{(i)}  \simeq \norm{g^{(i)}}.
\label{eq:relRMSE2}
\end{equation}

\begin{remark}
	For the first V-cycle, no estimation of the gradient norm may be available. A reasonably large value of $\epsilon^{(1)}$ should then be provided by the user.
	A large value of $\epsilon^{(1)}$ has almost no effect on the total optimization time; it results in a few additional very cheap iterations to attain $\epsilon^{(i)} \leq \cl\norm{g^{(i)}}$. However, if $\epsilon^{(1)}$ is much smaller than the starting gradient norm, the first iterations are significantly more expensive than necessary, until the gradient norm has caught up with $\epsilon^{(i)}$.
\end{remark}

We now turn to the rudimentary stopping criterion specified in Lines $5$--$8$.
First $\smash{\norm{g^{(i)}}}$ is checked for convergence in Line $4$. 
This corresponds to checking the gradient of the iterate $\smash{v_K^{(i)}}$ evaluated using the same sample sets it was obtained with in the preceding V-cycle in Line $3$. 
This test is essentially free, as $\smash{\|g^{(i)}\|}$ was already obtained.
However, the most relevant performance metric is how well an iterate performs when checked against new samples. This is the purpose of Line $5$. Setting RMSE $\epsilon = r\tau$, we determine new values for $\vec{n}$ using (\ref{eq:MLMC_n}), and we choose new sample sets. With those sample sets, a single gradient $\smash{\nabla \rJ_{K}(v_K^{(i)})}$ is calculated in the tentative solution $\smash{v_K^{(i)}}$. This test is expensive as it costs a gradient calculation. The purpose of first executing Line $4$ is of course to make sure that the expensive test of Line $5$ is only attempted for iterates that actually have a chance at succeeding. 

\subsection{Cost}
\label{sec:opt/cost}
We discuss the cost as a function of the requested gradient tolerance $\tau$ under the assumptions that follow.
\begin{assumption}
	The norm of the gradient converges at least linearly, i.e., 
	$\norm{g^{(i)}} \lesssim \mu^{i}$ for some $\mu\in (0,1)$
	\label{as:linearconvergence}
\end{assumption}
This does not follow trivially from the linear convergence of, e.g., the smoothers or the full MG/OPT V-cycle, since the sample sets change from V-cycle to V-cycle. During each V-cycle a slightly different problem is thus considered. Theoretical convergence results exist for certain optimization algorithms that are fed imprecise gradient information. However, since we have not made any detailed assumptions regarding the underlying optimization scheme, this is included as an assumption in this paper.

\begin{assumption}
	A single triple $\rho, \phi, \kappa$ exists such that the assumptions in Theorem \ref{theorem:MLMC} hold uniformly, i.e., for the same constants implicit in $\lesssim$ in a neighborhood of the optimal point.
	\label{as:uniformity}
\end{assumption}
This can be considered a smoothness assumption. The assumption follows if the implicit constants do not diverge in a neighborhood of the optimal point. Each worst case constant can then be taken.
For $\phi$ in particular, the sequence (as a function of iteration $i$) of variances $V_\ell$ can be assumed to tend towards the variances $V_\ell$ at the optimal point for any $\ell$. This is observed in experiments in \cite{vanbarel2019robust}. 
From Assumption \ref{as:uniformity}, it follows that a single constant implicit in $\lesssim$ in expression (\ref{eq:MLMC_cost}) for the cost in Theorem \ref{theorem:MLMC} exists for all gradient calculations in a neighborhood of the optimal point.

The finest MG/OPT level $K$ is fixed when Algorithm \ref{alg:opt} starts. We must assume that it is chosen such that the bias (\ref{eq:MLMC_bias}) is small enough throughout the algorithm. This means that $K=L$ must depend on $\tau$ in such a way that the following assumption is satisfied:
\begin{assumption}
	The finest MG/OPT level $K$, equal to the finest MLMC level $L$, is chosen such that $\tau^2 \lesssim \norm{\mean{\chlvl{L}{}Q_L - Q}}^2 \leq r(1-\theta)\tau^2$ in a neighborhood of the optimizer $u$, where $\theta$ is the parameter described in \S\ref{sec:MLMC/MSE}.
	\label{as:K}
\end{assumption}
Following from the second remark under Theorem \ref{theorem:MLMC}, this assumption ensures that Theorem \ref{theorem:MLMC} holds for the gradient calculations where the RMSE tolerance $\epsilon^{(i)} = r\tau$, i.e., for the gradient calculations in the final V-cycle(s). At this point we have not discussed how $K$ must be determined in practice. Some possibilities of making the algorithm choose $K$ adaptively are discussed at the end of this section.

\begin{theorem}[Optimization cost]
	\label{theorem:MLMC_optcost}
	Let Assumptions \ref{as:linearconvergence}, \ref{as:uniformity}, and \ref{as:K} hold and let $\epsilon^{(1)}$ be independent of $\tau$.  
	The cost $\mathcal{C}_\text{opt}(\tau)$ for Algorithm \ref{alg:opt} to reach a gradient $\norm{g^{(k)}} \leq \tau$ is then
	\begin{equation}
	\mathcal{C}_\text{opt}(\tau) \lesssim 
	\left\{\begin{array}{ll}
	\tau^{-2} &\text{if } \phi > \kappa, \\
	\tau^{-2}(\log \tau)^2 &\text{if } \phi = \kappa, \\
	\tau^{-2-(\kappa-\phi)/\rho} &\text{if } \phi < \kappa, \\
	\end{array}\right. 
	\label{eq:tau_cost}
	\end{equation} 
	i.e., proportional to the cost of a single MLMC gradient evaluation with RMSE tolerance $\tau$. 
\end{theorem}
\begin{proof}
	The total cost is
	\begin{equation}
		\mathcal{C}_\text{opt}(\tau) = \sum_{i=0}^I C_V^{(i)}
	\end{equation}
	with $C_V^{(i)}$ the cost of V-cycle $i$, which is proportional to the cost $\mathcal{C}_{\text{g},K}^{(i)}$ of a gradient evaluation at level $K$; see (\ref{eq:costV}) and (\ref{eq:costValt}). 
	The cost $\mathcal{C}_{\text{g},K}^{(I)}$ of the gradient evaluations in the final V-cycle(s) is, due to Assumptions \ref{as:uniformity} and \ref{as:K}, given by Theorem \ref{theorem:MLMC}. Line \ref{line:eps} ensures $r\tau < \epsilon^{(i)}$, leading to
	\begin{equation}
	\mathcal{C}_{\text{g},K}^{(I)} \lesssim 
	\left\{\begin{array}{ll}
	\tau^{-2} &\text{if } \phi > \kappa, \\
	\tau^{-2}(\log \tau)^2 &\text{if } \phi = \kappa, \\
	\tau^{-2-(\kappa-\phi)/\rho} &\text{if } \phi < \kappa. \\
	\end{array}\right. 
	\end{equation}
	From (\ref{eq:MLMC_n}) and \S\ref{sec:MLMC/cost}, the cost of the other gradient evaluations depends on the RMSE $\epsilon^{(i)}$ as $\mathcal{C}_{\text{g},K}^{(i)} \simeq (\epsilon^{(i)}/\epsilon^{(I)})^{-2}\mathcal{C}_{\text{g},K}^{(I)}$ and thus from (\ref{eq:relRMSE2}), as $\mathcal{C}_{\text{g},K}^{(i)} \simeq  (\norm{g^{(i)}}/\norm{g^{(I)}})^{-2}\mathcal{C}_{\text{g},K}^{(I)}$. Finally, using Assumption \ref{as:linearconvergence}, one obtains
	\begin{equation*}
		\mathcal{C}_\text{opt}(\tau) = \sum_{i=0}^I C_V^{(i)} 
		\simeq \sum_{i=0}^I \mathcal{C}_{\text{g},K}^{(i)} 
		\simeq \sum_{i=0}^I \left(\frac{\norm{g^{(i)}}}{\norm{g^{(I)}}}\right)^{-2}\mathcal{C}_{\text{g},K}^{(I)}
		\simeq \sum_{i=0}^I \mu^{2I-2i}\mathcal{C}_{\text{g},K}^{(I)}
		\simeq \mathcal{C}_{\text{g},K}^{(I)}
	\end{equation*}
	since $\mu \in (0,1)$.
\end{proof}

\subsection{Full multigrid and adaptive choice of K}
The full multigrid (FMG) or nested iteration scheme consists of a succession of V-cycles, where the first V-cycle considers only the coarsest level and each subsequent V-cycle considers an additional finer level. Instead of the interpolation method used within the V-cycle itself, a higher order interpolation may be performed to map the result of the previous V-cycle to the finer level \cite{briggs2000multigrid}. The early V-cycles are much cheaper and serve to provide a good initial value for the next V-cycle. In Algorithm \ref{alg:opt}, early V-cycles are already cheaper since the requested tolerance on the gradient $\epsilon$ becomes stricter as the optimization progresses. Algorithm \ref{alg:opt} can in fact be considered to contain some FMG properties already. Since the cost of the last few V-cycles is comparable to the cost of all previous V-cycles (see \S\ref{sec:opt/cost}), the computational time is not expected to improve significantly by using FMG in this setting.

One advantage of an FMG-like implementation is the possibility of choosing $K$ adaptively. One could start with a small value for $K$. During each V-cycle, the bias is estimated. This can easily be done based on the differences calculated in the MLMC method \cite{cliffe2011, vanbarel2019robust}. If the bias is not small enough, i.e., if $\norm{\mean{\chlvl{K}{}Q_K - Q}}^2 \leq r(1-\theta)\tau^2$ does not hold, $K$ is incremented in the next V-cycle. Such an adaptive FMG-like scheme allows Assumption \ref{as:K} to be satisfied automatically.

\section{Numerical results}
\label{sec:numresults}
The proposed strategy is tested on three model problems. The first model problem consists of the Laplace equation where the right hand side (the source term) can be controlled at any point in the domain. This is the ubiquitous model problem analyzed in many papers, e.g., \cite{Ali2016, chen2014weighted, geiersbach2019stochastic, guth2019quasi, martin2019multilevel, martin2018analysis, Rosseel2012, vanbarel2019robust}. In the second problem, the flux at the boundary of the Laplace PDE is to be controlled by the Dirichlet condition at that boundary. This model problem is more challenging for an optimization algorithm that uses only finest level information \cite{lewis2005}. The last problem consists of the viscous Burgers' PDE, where the state at some specified end time is to be controlled by the starting condition. Burgers' equation is often investigated as a first test problem for algorithms in fluid dynamics, and is relevant in assessing the performance for more realistic problems. Similar problems were investigated in this context in \cite{kouri2013trust, Kouri2014}. In all three cases, the uncertainty can be found in the diffusion coefficients, whose nature is discussed shortly. 

We compare these with the method presented in \cite{vanbarel2019robust} that utilizes just a single optimization level on which MLMC obtains the gradient, as illustrated in Figure \ref{fig:class1}. One can summarize the algorithm as follows. It executes the NCG method until the gradient becomes smaller than the RMSE it is calculated with, i.e., until (\ref{eq:relRMSE}) no longer holds. New samples are then taken with a RMSE multiplied by $0.25$. 
The details are more or less as presented in \cite{vanbarel2019robust}, with some slight differences to make the comparison as fair as possible; e.g., the definition of the RMSE is pointwise in \cite{vanbarel2019robust}, but is done using an $L^2$-norm in this paper; see \S\ref{sec:MLMC/MSE}. Both the equivalent number of fine grid gradient evaluations and the wall-clock time are compared. Each experiment starts with a guess of zero for the control.

The comparison with the algorithm classes represented by Figures \ref{fig:class3} and \ref{fig:class4} is not performed, since, for the high-dimensional uncertainties considered here, using a MC method would result in a large number of samples on the finest level. The methods in those classes usually use a more sophisticated quadrature technique, different from MC. This makes them likely better suited for problems with a smaller stochastic dimension, in which case the `$4$' in Figures \ref{fig:class3} and \ref{fig:class4} denotes a reasonably small number of solves. For low-dimensional uncertainties, well implemented methods of those types have a good chance to outperform the MC based method described in this paper. 

All calculations and timings are performed in Julia 1.3.1 and executed on an Intel Core\texttrademark \hspace{0.1cm}i5-4690K CPU @ 3.50 GHz. The software that allows the reproduction of these results can be made available on request.

\subsection{Stochastic field}
\newcommand{\sfield}{z}
\newcommand{\nkl}{ {n_{\text{KL}} } }
\newcommand{\cov}[2]{\text{Cov}\left[#1,#2\right]}
The uncertainties manifest themselves as stochastic fields $k$, which are taken here to be lognormal, i.e., $k(x,\omega) \triangleq \exp(z(x,\omega))$, with $z$ a Gaussian field. As in \cite{cliffe2011, graham2011}, e.g., we assume $\mean{z(x,\omega)} = 0$ and an exponential covariance 
\begin{equation}
C(x,x') = \cov{\sfield(x,\omega)}{\sfield(x',\omega)} = \sigma^2 \exp\Big(-\frac{\norm{x-x'}}{\lambda}\Big)
\label{eq:covariance_Gaussian}
\end{equation} 
with $\sigma^2$ the variance of the field, $\lambda$ the correlation length, and $\norm{\und}$ the $2$-norm in $\mathbb{R}^d$. A typical realization of $k$ can be found in Figure \ref{fig:sample1}.

One way to generate samples of $z$ is to use the Karhunen-Lo\`eve (KL) expansion \cite{loeve1946fonctions, karhunen1947ueber}
\begin{equation}
\sfield(x,\omega) = \mean{\sfield(x,\omega)} + \sum\limits_{n=1}^{\infty} \sqrt{\theta_n} \xi_n(\omega) f_n(x). \label{eq:KLexpansion}
\end{equation}
The KL expansion is the unique expansion of the above form that minimizes the total MSE if the expansion is truncated to a finite number of terms 
\cite{ghanem2003}. This sampling method is used in, e.g., \cite{borzi2009VWmultigrid, borzi2011pod, chen2014weighted, cliffe2011, graham2011}. The KL expansion represents the field at all points in the domain $D$, but one must be aware of the truncation error. 

\newcommand{\normals}{\bs{\xi}}
\newcommand{\samples}{\bs{z}}
Alternatively, one can generate exact realizations of the field in a finite set of $m$ discretization points $x_1, \ldots, x_m$ by considering the resulting covariance matrix $\Sigma = (C(x_i,x_j))_{i,j=1}^m$ and a factorization of the form $\Sigma = LL^T$, such as the Cholesky factorization. Clearly, samples $\samples = L\normals$ with $\normals$ a vector of independent standard normal random variables, then have the desired covariance: 
$$\mean{\samples\samples^T} = \mean{L\normals\normals^T L^T} = L\mean{\normals\normals^T}L^T = LL^T = \Sigma.$$ One may have observed that the covariance (\ref{eq:covariance_Gaussian}) of the stochastic field is homogeneous, meaning that it is a function of $x-x'$ only. In this case, the circulant embedding method \cite{dietrich1997,	graham2011} can be used to very efficiently sample the stochastic field in a regular rectangular grid of points in $\mathbb{R}^2$. $\Sigma$ is then block-Toeplitz with Toeplitz blocks and can be embedded\footnote{Some padding is sometimes necessary to ensure positive definiteness; for details, see, e.g. \cite{dietrich1997,graham2011}.} in a block-circulant matrix with circulant blocks (hence the name of the method), which can be factorized using a multidimensional FFT. Generation of $L\normals$ can then be performed in $\mathcal{O}(m\log m)$ operations. For $d>2$, this generalizes as one would expect. This last method is used to sample the stochastic fields in this paper.
\begin{figure}[h]
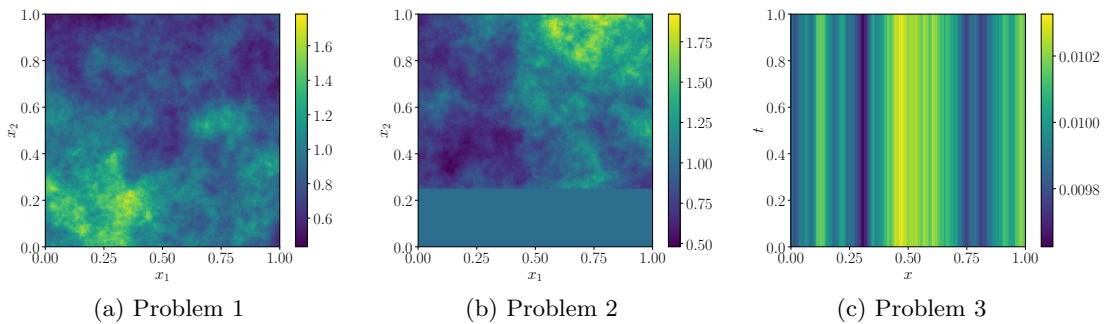

	\centering
	\begin{subfigure}{.32\textwidth}
		\centering
		\scalebox{0.33}{\hspace{-1.75cm}\input{Figure_1.pgf}}
		\caption{Problem 1 \hphantom{......}}
		\label{fig:sample1}
	\end{subfigure}
	\begin{subfigure}{.32\textwidth}
		\centering
		\scalebox{0.33}{\hspace{-1.20cm}\input{Figure_4.pgf}}
		\caption{Problem 2 \hphantom{..}}
		\label{fig:sample2}
	\end{subfigure}
	\begin{subfigure}{.32\textwidth}
		\centering
		\scalebox{0.33}{\hspace{-0.65cm}\input{Figure_6.pgf}}
		\caption{Problem 3}
		\label{fig:sample3}
	\end{subfigure}
	\caption{Samples of the stochastic fields used in each of the model problems, calculated using the circulant embedding method. Problem 1: A realization of the lognormal field specified by (\ref{eq:covariance_Gaussian}) with $\lambda = 0.3$ and $\sigma^2 = 0.1$. Problem 2: Another realization of the same random field, but for $x_2 \in [0,1/4]$, the field is set to $1$. Problem 3: A realization of the 1-dimensional lognormal field specified by (\ref{eq:covariance_Gaussian}) with $\lambda = 0.3$ and $\sigma^2 = 0.1$, multiplied by a scaling factor of $\num{1e-3}$.}
	\label{fig:samples}
\end{figure}

\subsection{Problem 1: Laplace equation}
Consider the Laplace equation on $D = (0,1)^2 \subset \mathbb{R}^2$ with boundary $\boundary{D}$. 
The state $y$ can be interpreted as a temperature distribution on $D$. A heat source (or sink) on $D$ constitutes the control $u$.
Problem 1 then consists of the robust optimization problem (\ref{eq:robust_problem}),
\begin{equation*}
	\min_{u \in U, y \in Y} J(y,u) = \frac{1}{2} \mean{\norm{y-z}^2} + \frac{\alpha}{2}\norm{u}^2 \quad \text{s.t.}\quad c(y,u)=0,
\end{equation*}
where the constraint $c$ is given by
\begin{align}
\begin{aligned}
-\nabla\cdot(k(x,\omega)\nabla y(x,\omega)) &= u(x) && \mbox{on } D, \\
y(x,\omega) &= 0 && \mbox{on } \boundary{D}.
\end{aligned}
\label{model_SPDE}
\end{align}
Here $u \in U = L^2(D)$ and $y \in Y = H_0^1(D) \otimes L^2(\Omega)$. 
The cost is quadratic and the constraint $c$ is linear, and therefore the problem is quadratic and convex, guaranteeing a unique solution.
We take $\alpha = \num{1e-6}$ and define the target function as
$$
z(x) = \begin{cases} 
1 & \text{if } x \in [1/4,3/4] \times [1/4,3/4], \\
0 & \text{otherwise}.
\end{cases}
$$
The uncertainty exists in the heat conduction coefficient, which is modeled as the lognormal stochastic field $k: D\times \Omega \rightarrow \mathbb{R}:(x,\omega) \mapsto k(x,\omega)$, defined, as explained previously, by (\ref{eq:covariance_Gaussian}) with $\lambda=0.3$ and $\sigma^2=0.1$. A sample of $k$ is shown in Figure \ref{fig:sample1}. 

\subsubsection{Optimality conditions}
Substituting the terms related to the constraint in (\ref{eq:optcond_lagr_general}) or (\ref{eq:optcond_lagr}) leads to the optimality conditions
\begin{equation}
\left\{\begin{array}{rcll}
-\nabla \cdot(k\nabla y)& = & u &\quad \mbox{on } D,\\
-\nabla \cdot(k\nabla p)& = &y - z  &\quad \mbox{on } D,\\
\nabla \rJ(u) & = & \alpha u + \mean{p} = 0 &\quad \mbox{on } D
\end{array}\right.
\end{equation}
with $y = p = 0$ on $\boundary{D}$.

\subsubsection{Discretization and optimization details}
The PDE is discretized using a finite difference method of second order, which, on these regular discretization grids, can also be interpreted as a finite element method. For all three model problems, the discretization of the optimality conditions is such that it corresponds to the optimality conditions of the discretized problem. 
For the implementation of the level mapping operator $\chlvl{\ell_1}{\ell_2}$, we use linear interpolation and full weighted coarsening, i.e., interpolation and coarsening defined by the stencils
$$
\frac{1}{16}\begin{bmatrix}
1 & 2 & 1\\
2 & 4 & 2\\
1 & 2 & 1
\end{bmatrix} \quad \text{ and } \quad 
\frac{1}{4}\left]
\begin{matrix}
1 & 2 & 1\\
2 & 4 & 2\\
1 & 2 & 1
\end{matrix}
\right[,
$$
respectively. The inner products and norms are of course such that (\ref{eq:chlvl_adjoint}) holds.
The problem is quadratic; the smoother is the NCG smoother as described in \S\ref{sec:opt/smoother}\footnote{However, the gradient at each iteration point is calculated explicitly, i.e., not as a linear combination of previous gradients. This makes the counting more uniform across the three model problems, since Problem 3 has a nonlinear constraint, which does not allow for this shortcut to be used. Each smoothing step therefore contributes two gradient evaluations to the total cost.}. We take $K=L=4$ with a coarsest level of $17 \times 17$ and a finest level of $257 \times 257$. 

\subsubsection{Results}
Just 2 MG/OPT V-cycles were required to reach a gradient norm of $\num{5e-5}$. For the $i$-th V-cycle, Table \ref{tab:mgopt1} below shows the requested RMSE $\epsilon^{(i)}$, the number of samples $\vec{n}$ used for the finest optimization problem,
the cost and gradient norm before ($\smash{J^{(i)}_0}$ and $\smash{\norm{\gstart}}$) and after ($\smash{J^{(i)}}$ and $\smash{\norm{\gend}}$), calculated using that V-cycle's sample sets, the theoretical equivalent number of fine grid PDE solves and the measured wallclock time in seconds. Note that the starting cost and gradient in V-cycle $i+1$ is not exactly equal to the final gradient of V-cycle $i$ since the sample sets used in their calculations are different.
A calculation using new samples yields $J = \num{1.36e-02}$ and $\|\nabla J(u)\| = \num{4.47e-05}$. The total number of finest grid PDE solves is $637$, the total time is $590$ s.
\begin{table}[h]
\begin{equation*}
\arraycolsep=3.0pt
\small
\begin{array}{l|lllllllllllll}
i & \epsilon^{(i)} & n_0 & n_1 & n_2 & n_3 & n_4 & J^{(i)}_0 & J^{(i)} &\norm{\gstart} & \norm{\gend} & \text{Solves} & \text{Time [s]}  \\
\hline
1 & \num{1.00e-01} & 2106 & 440 & 92 & 20 & 4 & \num{1.27e-01} & \num{1.43e-02} & \num{2.09e-02} & \num{9.31e-05} & 132 & 211 & \\
2 & \num{2.50e-05} & 18609 & 998 & 209 & 41 & 6 & \num{1.41e-02} & \num{1.40e-02} & \num{1.48e-04} & \num{1.56e-05} & 505 & 379 & \\
\end{array}
\end{equation*}
\vspace{-0.4cm}
\caption{MG/OPT results for Problem 1.}
\label{tab:mgopt1}
\end{table}

We compare this to the results obtained using finest level optimization only, which required $26$ NCG steps. Table \ref{tab:ncg1} shows information about the NCG steps in which new sample sets for the underlying MLMC method were generated. The values for $J^{(i)}$ and $\norm{\gend}$ are those after iteration $i$ has completed. 
After $26$ iterations, a new sample test yields $J = \num{1.37e-02}$ and $\|\nabla J(u)\| = \num{4.82e-05}$.
The total number of finest grid PDE solves is $2854$, the total time is $1204$ s. 
\begin{table}[h]
\begin{equation*}
\arraycolsep=3.0pt 
\small
\begin{array}{l|llllllllll}
	i & \epsilon^{(i)} & n_0 & n_1 & n_2 & n_3 & n_4 & J^{(i)} & \norm{\gend} & \text{Solves} & \text{Time [s]} \\
	\hline
	0 & \num{1.00e-02} & 2106 & 440 & 92 & 20 & 4 & \num{1.27e-01} & \num{2.10e-02} & 47(\times 14) & 20 (\times 14) \\
	15 & \num{2.50e-05} & 20376 & 1160 & 237 & 47 & 13 & \num{1.42e-02} & \num{2.40e-04} & 183(\times 12) & 77 (\times 12) \\
\end{array}
\end{equation*}
\vspace{-0.4cm}
\caption{Results for Problem 1 using finest level optimization only.}
\label{tab:ncg1}
\end{table}

Plots of the results can be found in Figure \ref{fig:results1}. 
\begin{figure}[h]
	\hspace{-0.01\textwidth}
	\begin{subfigure}[t]{.35\textwidth}
		\centering
		\sethw{0.88}{0.88}
		\input{fig/prob1/p2mgopt/u.tex} \\
		\input{fig/prob1/p2mgopt/grad.tex} \\
	\end{subfigure}
	\hspace{-0.04\textwidth}
	\begin{subfigure}[t]{.35\textwidth}
		\centering
		\sethw{0.88}{0.88}
		\input{fig/prob1/p2mlmc/u.tex} \\
		\input{fig/prob1/p2mlmc/grad.tex} \\
	\end{subfigure}
	\hspace{-0.04\textwidth}
	\begin{subfigure}[t]{.35\textwidth}
		\centering
		\sethw{0.88}{0.88}
		\input{fig/prob1/p2mgopt/Ey.tex} \\
		\input{fig/prob1/p2mgopt/Vy.tex} \\
	\end{subfigure}
	\caption{Optimization results for Problem 1. The leftmost figures show the solution $u$ and its gradient $\nabla J(u)$ obtained using MG/OPT. The figures in the middle are obtained using finest level NCG only. The rightmost figures show $\mean{y}$ and $\var{y}$ at the optimal point. They look identical for both optimization methods. All figures are shown on the finest grid of $257 \times 257$.}
	\label{fig:results1}
\end{figure}
First, we can observe that the control input $u$ generated using only finest level optimization looks smoother. Yet, it does not seem to yield a lower cost function. Furthermore, the residual gradient contains mostly low frequency components, while the MG/OPT gradient contains high frequencies also. This is what is expected in a multigrid-like method such as MG/OPT \cite{briggs2000multigrid}. The smoother, as its name suggests, removes mostly high frequency components from the gradient, while it has more trouble removing low frequency components. The gradients in Figure \ref{fig:results1} suggest that the MG/OPT method is more successful at removing components across a wide frequency spectrum. These two observations can also be made for the other two test problems.

\subsection{Problem 2: Dirichlet to Neumann map}
The constraint consists again of the diffusion PDE, but the control now acts on a part $\Gamma = (0,1) \times \{0\}$ of the boundary $\boundary{D}$.
\begin{align} \label{eq:prob2_pde}
\begin{aligned}
-\nabla\cdot(k(x,\omega)\nabla y(x,\omega)) & = 0 && \mbox{on } D, \\
y(x,\omega) & = u(x) && \mbox{on } \Gamma, \\
y(x,\omega) & = 0 && \mbox{on } \boundary{D}\setminus \Gamma.
\end{aligned}
\end{align}
The goal is to bring the flux at $\Gamma$ as close as  possible to some target flux $\phi_\Gamma \in L^2(\Gamma)$ on $\Gamma$. To that end, we optimize
\begin{equation}
J(y,u) = \frac{1}{2}\mean{\int_\Gamma (k\diffp{y}{\vec{n}}-\phi_\Gamma)^2 \d{x}} + \frac{\alpha}{2}\int_\Gamma u^2 \d{x}.
\label{eq:p4.cont.costfun}
\end{equation}
This problem is also quadratic and convex. We take $\alpha = \num{1e-6}$ and define the target as
$\phi_\Gamma(x) = \sin(\pi x).$
The uncertainty exists again in $k$ and is of the same nature as in Problem 1, except in the region $[0,1] \times [0,1/4]$ that touches the boundary $\Gamma$. A sample of $k$ is shown in Figure \ref{fig:sample2}. Having a deterministic zone close to the control ensures that the assumption on the variance in Theorem \ref{theorem:MLMC} holds. The smaller the deterministic zone is, the longer it takes before the asymptotic decay of the variance with level begins. In practice, one must then take a finer coarsest level.

We consider controls $u \in U = L^2(\Gamma)$, because it is a convenient setting for numerical optimization. 
However, only the controls in $H^{1/2}(\Gamma) \subset L^2(\Gamma)$ produce state realizations in $H^1(D)$. Therefore, we take the larger state space $Y=L^2(D) \otimes L^2(\Omega)$. The solution of (\ref{eq:prob2_pde}) for a given boundary in $L^2(\Gamma) \setminus H^{1/2}(\Gamma)$ could then be understood as the solution to the so-called very weak form of the PDE (\ref{eq:prob2_pde}) \cite{may2013error}. Some alternative formulations and corresponding function space settings for this type of problem are discussed in \cite{kunisch2007constrained}.

\subsubsection{Optimality conditions}
Evaluating (\ref{eq:optcond_lagr_general}) for this specific problem leads to the optimality conditions
\begin{equation}
\left\{\begin{array}{rcll}
-\nabla \cdot(k\nabla y) &=& 0 &\quad \mbox{on } D,\\
-\nabla \cdot(k\nabla p) &=& 0  &\quad \mbox{on } D,\\
\nabla \rJ(u) = \alpha u + k\diffp{p}{\vec{n}} &=& 0 &\quad \mbox{on } \Gamma
\end{array}\right.
\end{equation}
with $p= 0$ on $\boundary{D}\setminus \Gamma$ and $p= k\diffp{y}{\vec{n}} - \phi_\Gamma$ on $\Gamma$ and $y=0$ on $\boundary{D}$.

\subsubsection{Discretization and optimization details}
The discretization method is the same as in Problem 1. As in all three problems, the discretization of the optimality conditions is such that it corresponds to the optimality conditions of the discretized problem. 
Even though we adopt the setting $U=L^2(\Gamma)$ and $Y=L^2(D) \otimes L^2(\Omega)$, the intermediate controls and states that arise during the optimization process and those constituting the final solution have additional regularity. One reason is the first term in (\ref{eq:p4.cont.costfun}), which drives the control $u$ in a direction that constrains some of the partial derivatives of the state. While the intermediate controls and states are not shown, one can observe that the final solution is indeed smooth in Figure \ref{fig:results2}. The discretization should therefore not run into problems here. 

We still use linear interpolation and full weighted coarsening. Since the space $L^2(\Gamma)$ in which $u$ resides now contains functions defined on the 1-dimensional $\Gamma$ instead of on $D$, the corresponding stencils are
$$
\frac{1}{4}\begin{bmatrix}
1 & 2 & 1
\end{bmatrix} \quad \text{ and } \quad 
\frac{1}{2}\left]
\begin{matrix}
1 & 2 & 1\\
\end{matrix}
\right[.
$$
Again it is important that the inner products and norms satisfy (\ref{eq:chlvl_adjoint}). The problem is quadratic and the smoother is the same as in Problem 1. We take $K=L=5$ with a coarsest level of $9 \times 9$ and a finest level of $257 \times 257$. 

\subsubsection{Results}
The solution is shown in Figure \ref{fig:results2}.
Algorithm \ref{alg:MG/OPT} requires $6$ V-cycles to obtain a target gradient norm of $\num{1e-3}$, see Table \ref{tab:mgopt2}.
The repeating numbers of samples correspond to the numbers of warmup samples used to estimate the variance. A calculation using new samples yields $J = \num{2.77e-04}$ and $\|\nabla J(u)\| = \num{8.35e-04}$. The total number of finest grid PDE solves is $1626$, the total time is $1661$ s.
\begin{table}[h]
\begin{equation*}
\arraycolsep=3.0pt
\small
\begin{array}{l|llllllllllllll}
i & \epsilon^{(i)} & n_0 & n_1 & n_2 & n_3 & n_4 & n_5 & J^{(i)}_0 & J^{(i)} &\norm{\gstart} & \norm{\gend} & \text{Solves} & \text{Time [s]}  \\
\hline
1 & \num{1.00e-01} & 250 & 250 & 250 & 92 & 20 & 4 & \num{2.48e-01} & \num{2.36e-03} & \num{2.21e+00} & \num{2.30e+00} & 102 & 154\\
2 & \num{5.76e-01} & 250 & 250 & 250 & 92 & 20 & 4 & \num{2.15e-03} & \num{2.77e-05} & \num{2.30e+00} & \num{3.89e-01} & 102 & 183\\
3 & \num{3.29e-02} & 250 & 250 & 250 & 92 & 20 & 4 & \num{1.01e-03} & \num{8.69e-04} & \num{3.90e-01} & \num{6.69e-02} & 102 & 148\\
4 & \num{5.74e-03} & 705 & 250 & 250 & 92 & 20 & 4 & \num{5.64e-04} & \num{5.47e-04} & \num{6.90e-02} & \num{1.64e-02} & 103 & 153\\
5 & \num{1.96e-03} & 4751 & 930 & 250 & 92 & 20 & 4 & \num{1.66e-04} & \num{1.61e-04} & \num{1.86e-02} & \num{4.13e-03} & 121 & 170\\
6 & \num{5.00e-04} & 67656\hspace{-1cm} & 12604 & 3185 & 699 & 173 & 22 & \num{2.08e-04} & \num{2.07e-04} & \num{4.22e-03} & \num{7.59e-04} & 1096 & 853\\
\end{array}
\end{equation*}
\vspace{-0.4cm}
\caption{MG/OPT results for problem 2.}
\label{tab:mgopt2}
\end{table}

Using only the finest optimization level, the target gradient norm of $1e-3$ is not reached after $200$ iterations. A new sample test yields $J = \num{2.79e-04}$ and $\|\nabla J(u)\| = \num{2.16e-03}$. The total equivalent number of fine grid solves came to $22422$, the total time to $10764$ s. As before, the NCG steps in which new sample sets were generated are detailed in Table \ref{tab:ncg2} below.
\begin{table}[h]
\begin{equation*}
\arraycolsep=3.0pt
\small
\begin{array}{l|lllllllllll}
i & \epsilon^{(i)} & n_0 & n_1 & n_2 & n_3 & n_4 & n_5 & J^{(i)} & \norm{\gend} & \text{Solves} & \text{Time [s]} \\
\hline
0 & \num{1.00e-02} & 250 & 250 & 250 & 92 & 20 & 4 & \num{2.48e-01} & \num{2.20e+00} & 36(\times 114) & 22(\times 114) \\
115 & \num{7.22e-04} & 33601 & 6869 & 1532 & 415 & 68 & 18 & \num{2.03e-04} & \num{6.59e-03} & 213(\times 86) & 96(\times 86) \\
\end{array}
\end{equation*}
\vspace{-0.4cm}
\caption{Results for Problem 2 using finest level optimization only.}
\label{tab:ncg2}
\end{table}
\begin{figure}[h]
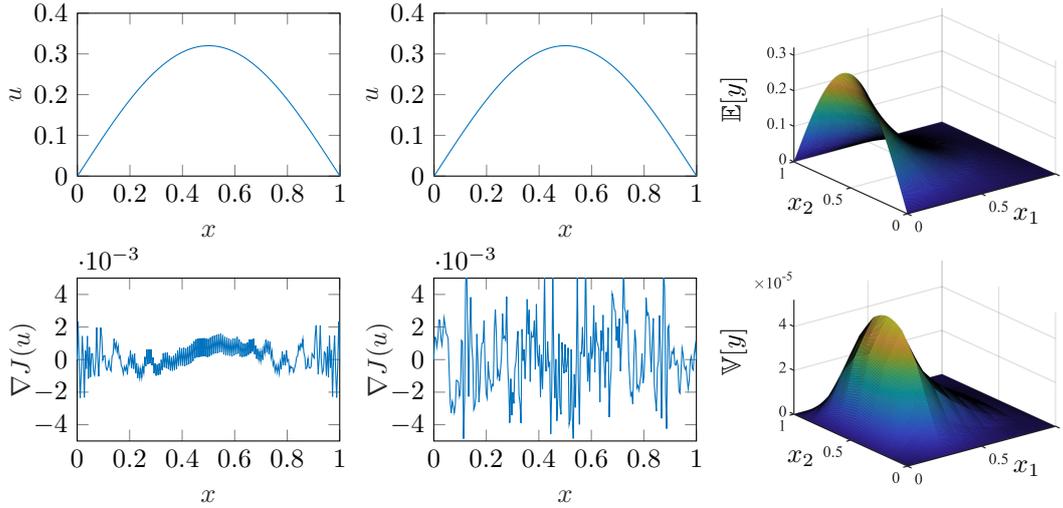
 
	\hspace{-0.01\textwidth}
	\begin{subfigure}[t]{.30\textwidth}
		\centering
		\sethw{0.5}{0.8}
		\definecolor{mycolor1}{rgb}{0.00000,0.44700,0.74100}%
\definecolor{mycolor2}{rgb}{0.85000,0.32500,0.09800}%
\definecolor{mycolor3}{rgb}{0.92900,0.69400,0.12500}%
\definecolor{mycolor4}{rgb}{0.49400,0.18400,0.55600}%
\definecolor{mycolor5}{rgb}{0.46600,0.67400,0.18800}%
\begin{tikzpicture}

\begin{axis}[%
width=\figurewidth,
height=\figureheight,
at={(0\figurewidth,0\figureheight)},
scale only axis,
xmin=0,
xmax=1,
xlabel style={font=\color{white!15!black}},
xlabel={$x$},
ylabel style={font=\color{white!15!black}, at={(0.12,0.5)}},
ylabel={$u$},
ymin=0,
ymax=0.4,
yminorticks=true,
]
\addplot [color=mycolor1]
table[row sep=crcr]{%
0.0  			0.0\\
0.00390625		0.003764904039830596\\
0.0078125  		0.007602131337257026\\
0.01171875  	0.011452943729685765\\
0.015625  		0.015312316259087992\\
0.01953125  	0.01917709343691863\\
0.0234375  		0.023045035816112436\\
0.02734375  	0.026914593520675892\\
0.03125  0.030784610498813897 \\
0.03515625  0.03465371747100151 \\
0.0390625  0.03852103796997834 \\
0.04296875  0.04238546065807885 \\
0.046875  0.04624632888959038 \\
0.05078125  0.05010263283701204 \\
0.0546875  0.05395374595326201 \\
0.05859375  0.05779879712794444 \\
0.0625  0.061637056877137525 \\
0.06640625  0.06546786755647181 \\
0.0703125  0.06929054026131122 \\
0.07421875  0.07310437427388398 \\
0.078125  0.07690855821891701 \\
0.08203125  0.08070255216923579 \\
0.0859375  0.08448576653254125 \\
0.08984375  0.08825746828581993 \\
0.09375  0.09201695700214732 \\
0.09765625  0.09576363376897203 \\
0.1015625  0.09949695175558969 \\
0.10546875  0.10321625240599987 \\
0.109375  0.10692091459909879 \\
0.11328125  0.11061036698466988 \\
0.1171875  0.11428400986642491 \\
0.12109375  0.11794123778158964 \\
0.125  0.12158143406324451 \\
0.12890625  0.12520403804041308 \\
0.1328125  0.12880848760499033 \\
0.13671875  0.13239415270162644 \\
0.140625  0.13596048567255262 \\
0.14453125  0.1395068535768283 \\
0.1484375  0.14303274937495292 \\
0.15234375  0.1465375734886788 \\
0.15625  0.15002081298649445 \\
0.16015625  0.15348187106979058 \\
0.1640625  0.15692025150960448 \\
0.16796875  0.16033538037865627 \\
0.171875  0.16372675205348627 \\
0.17578125  0.1670938221548136 \\
0.1796875  0.17043608778343453 \\
0.18359375  0.17375298690355295 \\
0.1875  0.17704402371216849 \\
0.19140625  0.18030865396220638 \\
0.1953125  0.18354641752736722 \\
0.19921875  0.18675673964608727 \\
0.203125  0.1899391648332301 \\
0.20703125  0.1930931250245519 \\
0.2109375  0.1962182183850383 \\
0.21484375  0.19931387627010125 \\
0.21875  0.20237969402702152 \\
0.22265625  0.20541510291752638 \\
0.2265625  0.20841972683001347 \\
0.23046875  0.21139302606601132 \\
0.234375  0.21433461022335432 \\
0.23828125  0.21724396039317215 \\
0.2421875  0.22012067886612882 \\
0.24609375  0.22296426707024042 \\
0.25  0.22577433499888688 \\
0.25390625  0.22855039204087294 \\
0.2578125  0.23129205777466225 \\
0.26171875  0.2339988357447424 \\
0.265625  0.23667037240931021 \\
0.26953125  0.23930617040758786 \\
0.2734375  0.24190590564493594 \\
0.27734375  0.24446909288697535 \\
0.28125  0.24699541373100822 \\
0.28515625  0.249484395469315 \\
0.2890625  0.2519357306774271 \\
0.29296875  0.2543489790135629 \\
0.296875  0.2567238398122114 \\
0.30078125  0.25905988169553645 \\
0.3046875  0.2613568066304329 \\
0.30859375  0.26361419773712336 \\
0.3125  0.2658317643949075 \\
0.31640625  0.2680091011909468 \\
0.3203125  0.2701459374168745 \\
0.32421875  0.2722418703473759 \\
0.328125  0.2742966443945755 \\
0.33203125  0.2763098584959199 \\
0.3359375  0.27828128319980416 \\
0.33984375  0.2802105337543328 \\
0.34375  0.282097394978173 \\
0.34765625  0.28394148935334795 \\
0.3515625  0.2857426186663921 \\
0.35546875  0.2875004257959868 \\
0.359375  0.28921471875965304 \\
0.36328125  0.2908851569270243 \\
0.3671875  0.2925115579762258 \\
0.37109375  0.294093592089341 \\
0.375  0.2956310947965895 \\
0.37890625  0.2971237432952542 \\
0.3828125  0.2985713946981014 \\
0.38671875  0.2999737335838711 \\
0.390625  0.30133063164124263 \\
0.39453125  0.3026417853399617 \\
0.3984375  0.30390708951636713 \\
0.40234375  0.3051262540387611 \\
0.40625  0.3062991871456396 \\
0.41015625  0.30742560896296217 \\
0.4140625  0.3085054458714476 \\
0.41796875  0.30953843512418006 \\
0.421875  0.3105245109436007 \\
0.42578125  0.31146342477070244 \\
0.4296875  0.3123551281080265 \\
0.43359375  0.3131993846258082 \\
0.4375  0.31399616311452555 \\
0.44140625  0.31474523978699426 \\
0.4453125  0.31544659832857314 \\
0.44921875  0.3161000277240022 \\
0.453125  0.31670552575978816 \\
0.45703125  0.3172628980835015 \\
0.4609375  0.31777216048734286 \\
0.46484375  0.3182331321425427 \\
0.46875  0.31864584529262935 \\
0.47265625  0.31901013405146333 \\
0.4765625  0.3193260434827567 \\
0.48046875  0.31959342353078257 \\
0.484375  0.3198123329792223 \\
0.48828125  0.31998263743252686 \\
0.4921875  0.320104411536821 \\
0.49609375  0.32017753648678005 \\
0.5  0.3202021023040973 \\
0.50390625  0.3201780048163441 \\
0.5078125  0.3201053476280319 \\
0.51171875  0.31998404016180243 \\
0.515625  0.31981420075762623 \\
0.51953125  0.31959575422541087 \\
0.5234375  0.3193288343506213 \\
0.52734375  0.3190133822370338 \\
0.53125  0.3186495470925152 \\
0.53515625  0.318237284026345 \\
0.5390625  0.31777675803814837 \\
0.54296875  0.3172679370433933 \\
0.546875  0.3167110006983984 \\
0.55078125  0.31610593279868054 \\
0.5546875  0.31545292740823944 \\
0.55859375  0.31475198668137083 \\
0.5625  0.31400332096742 \\
0.56640625  0.31320694653002956 \\
0.5703125  0.3123630862393135 \\
0.57421875  0.3114717714642085 \\
0.578125  0.31053323677472255 \\
0.58203125  0.30954753101497823 \\
0.5859375  0.308514900928668 \\
0.58984375  0.30743541455958645 \\
0.59375  0.3063093321813199 \\
0.59765625  0.3051367298567211 \\
0.6015625  0.30391788393257035 \\
0.60546875  0.3026528888128785 \\
0.609375  0.30134203035615065 \\
0.61328125  0.2999854161688545 \\
0.6171875  0.29858334678490067 \\
0.62109375  0.29713595349362726 \\
0.625  0.2956435490751142 \\
0.62890625  0.29410627921235183 \\
0.6328125  0.2925244629324173 \\
0.63671875  0.29089826742459907 \\
0.640625  0.28922801767382983 \\
0.64453125  0.2875138997095693 \\
0.6484375  0.2857562492856378 \\
0.65234375  0.28395526460605813 \\
0.65625  0.2821112974492511 \\
0.66015625  0.2802245519898268 \\
0.6640625  0.27829539912800016 \\
0.66796875  0.276324060100591 \\
0.671875  0.27431091139455926 \\
0.67578125  0.27225618889430436 \\
0.6796875  0.2701602875930251 \\
0.68359375  0.2680234703159744 \\
0.6875  0.2658461331480307 \\
0.69140625  0.26362855273106434 \\
0.6953125  0.2613711285299945 \\
0.69921875  0.2590741591258647 \\
0.703125  0.2567380531132231 \\
0.70703125  0.25436311747883544 \\
0.7109375  0.2519497752527735 \\
0.71484375  0.24949833858713177 \\
0.71875  0.2470092373757279 \\
0.72265625  0.2444827904722273 \\
0.7265625  0.24191946035981657 \\
0.73046875  0.23931957688638067 \\
0.734375  0.23668361353906275 \\
0.73828125  0.23401190597297483 \\
0.7421875  0.2313049418189438 \\
0.74609375  0.22856308730665081 \\
0.75  0.2257868270448555 \\
0.75390625  0.22297655421224688 \\
0.7578125  0.2201327480244948 \\
0.76171875  0.21725581245956715 \\
0.765625  0.2143462325931136 \\
0.76953125  0.21140442010524685 \\
0.7734375  0.20843088130009035 \\
0.77734375  0.2054260225243875 \\
0.78125  0.20239036749250652 \\
0.78515625  0.1993243081904989 \\
0.7890625  0.19622839822525853 \\
0.79296875  0.19310305983872725 \\
0.796875  0.18994884545159707 \\
0.80078125  0.1867661728805123 \\
0.8046875  0.18355559510003555 \\
0.80859375  0.1803175844109391 \\
0.8125  0.17705269873906188 \\
0.81640625  0.17376141461811717 \\
0.8203125  0.1704442608793918 \\
0.82421875  0.1671017507779213 \\
0.828125  0.16373443011383834 \\
0.83203125  0.1603428174272014 \\
0.8359375  0.15692744219999244 \\
0.83984375  0.1534888267026186 \\
0.84375  0.15002752853177573 \\
0.84765625  0.14654405910418053 \\
0.8515625  0.14303900061553548 \\
0.85546875  0.13951288279962185 \\
0.859375  0.13596628997946425 \\
0.86328125  0.13239974368168836 \\
0.8671875  0.12881386280718052 \\
0.87109375  0.12520921009727973 \\
0.875  0.12158640084411874 \\
0.87890625  0.11794601083923029 \\
0.8828125  0.1142885871966603 \\
0.88671875  0.11061476046931855 \\
0.890625  0.10692512302627696 \\
0.89453125  0.10322028579566328 \\
0.8984375  0.09950080820816731 \\
0.90234375  0.09576732262059691 \\
0.90625  0.09202047630610487 \\
0.91015625  0.08826082524956137 \\
0.9140625  0.0844889590385251 \\
0.91796875  0.08070558736981262 \\
0.921875  0.07691143445210186 \\
0.92578125  0.07310709759163998 \\
0.9296875  0.06929310865975091 \\
0.93359375  0.06547028668988648 \\
0.9375  0.06163932487496065 \\
0.94140625  0.05780091835283409 \\
0.9453125  0.05395571852694825 \\
0.94921875  0.05010446095369779 \\
0.953125  0.0462480113740989 \\
0.95703125  0.042387000884585344 \\
0.9609375  0.0385224346418734 \\
0.96484375  0.034654973364049504 \\
0.96875  0.030785724578825648 \\
0.97265625  0.02691556763563843 \\
0.9765625  0.02304586902390739 \\
0.98046875  0.01917778735139835 \\
0.984375  0.015312870483958227 \\
0.98828125  0.011453359404208096 \\
0.9921875  0.007602408013608373 \\
0.99609375  0.003765042407123894 \\
1.0  0.0 \\
};
\end{axis}
\end{tikzpicture}%
 \\
		\input{fig/prob2/p1mgopt/grad.tex} \\
	\end{subfigure}
	\hspace{0.01\textwidth}
	\begin{subfigure}[t]{.30\textwidth}
		\centering
		\sethw{0.5}{0.8}
%
%
\definecolor{mycolor1}{rgb}{0.00000,0.44700,0.74100}%
\definecolor{mycolor2}{rgb}{0.85000,0.32500,0.09800}%
\definecolor{mycolor3}{rgb}{0.92900,0.69400,0.12500}%
\definecolor{mycolor4}{rgb}{0.49400,0.18400,0.55600}%
\definecolor{mycolor5}{rgb}{0.46600,0.67400,0.18800}%
\begin{tikzpicture}

\begin{axis}[%
width=\figurewidth,
height=\figureheight,
at={(0\figurewidth,0\figureheight)},
scale only axis,
xmin=0,
xmax=1,
xlabel style={font=\color{white!15!black}},
xlabel={$x$},
ylabel style={font=\color{white!15!black}, at={(0.12,0.5)}},
ylabel={$u$},
ymin=0,
ymax=0.4,
yminorticks=true,
axis background/.style={fill=white}
]
\addplot [color=mycolor1]
table[row sep=crcr]{%
0.0  0.0 \\
0.00390625  0.003765616745951403 \\
0.0078125  0.007603644908848854 \\
0.01171875  0.011455196265257827 \\
0.015625  0.015315200018024793 \\
0.01953125  0.019180567779170928 \\
0.0234375  0.02304912657016201 \\
0.02734375  0.026919258748114735 \\
0.03125  0.03078969582079495 \\
0.03515625  0.03465931092667902 \\
0.0390625  0.0385270416395092 \\
0.04296875  0.04239190998446385 \\
0.046875  0.04625307309575545 \\
0.05078125  0.05010972332562278 \\
0.0546875  0.05396106423656615 \\
0.05859375  0.05780631067860365 \\
0.0625  0.06164480146202801 \\
0.06640625  0.06547587929431962 \\
0.0703125  0.06929886108303275 \\
0.07421875  0.07311300140798523 \\
0.078125  0.0769176478407619 \\
0.08203125  0.08071216281676584 \\
0.0859375  0.08449588724830252 \\
0.08984375  0.08826809564432163 \\
0.09375  0.09202817716749452 \\
0.09765625  0.09577549547850008 \\
0.1015625  0.09950948032582214 \\
0.10546875  0.10322953033802905 \\
0.109375  0.10693486168888577 \\
0.11328125  0.11062482547422266 \\
0.1171875  0.11429920446323483 \\
0.12109375  0.11795728647786635 \\
0.125  0.12159814174893838 \\
0.12890625  0.1252211521779617 \\
0.1328125  0.12882592716802319 \\
0.13671875  0.13241212238252573 \\
0.140625  0.13597890105718127 \\
0.14453125  0.13952558412243293 \\
0.1484375  0.14305179817439687 \\
0.15234375  0.14655692655522623 \\
0.15625  0.15004057732770565 \\
0.16015625  0.1535021434168831 \\
0.1640625  0.1569411125078464 \\
0.16796875  0.16035669832651306 \\
0.171875  0.16374848582290286 \\
0.17578125  0.16711607370136727 \\
0.1796875  0.17045885838810193 \\
0.18359375  0.1737762409946279 \\
0.1875  0.17706769646218112 \\
0.19140625  0.18033273928949284 \\
0.1953125  0.18357073225840825 \\
0.19921875  0.1867813274962186 \\
0.203125  0.18996406031378504 \\
0.20703125  0.19311845614895576 \\
0.2109375  0.19624387291566336 \\
0.21484375  0.19933991078823776 \\
0.21875  0.20240607659511226 \\
0.22265625  0.2054419244406776 \\
0.2265625  0.2084469551658521 \\
0.23046875  0.2114207672527297 \\
0.234375  0.21436272884267926 \\
0.23828125  0.21727233555630057 \\
0.2421875  0.22014921073877036 \\
0.24609375  0.22299302613570582 \\
0.25  0.225803268811824 \\
0.25390625  0.22857955424595738 \\
0.2578125  0.23132141819514437 \\
0.26171875  0.23402844178978693 \\
0.265625  0.23670004605631764 \\
0.26953125  0.2393359487860889 \\
0.2734375  0.24193570814872217 \\
0.27734375  0.2444989764009438 \\
0.28125  0.24702512883516628 \\
0.28515625  0.24951402892310162 \\
0.2890625  0.25196526366404914 \\
0.29296875  0.25437859411206176 \\
0.296875  0.2567534163498761 \\
0.30078125  0.2590895579643316 \\
0.3046875  0.2613865302439491 \\
0.30859375  0.263644077173568 \\
0.3125  0.26586172318947804 \\
0.31640625  0.26803916126233096 \\
0.3203125  0.27017600872009484 \\
0.32421875  0.2722718434998161 \\
0.328125  0.2743265680584652 \\
0.33203125  0.2763398128333771 \\
0.3359375  0.27831142107817336 \\
0.33984375  0.28024070633889686 \\
0.34375  0.2821275910384945 \\
0.34765625  0.2839716853048654 \\
0.3515625  0.28577285444781714 \\
0.35546875  0.28753072832594057 \\
0.359375  0.2892448662821899 \\
0.36328125  0.2909151994742974 \\
0.3671875  0.2925412775100632 \\
0.37109375  0.2941231755454685 \\
0.375  0.2956604009138908 \\
0.37890625  0.297152780035069 \\
0.3828125  0.29860009361538686 \\
0.38671875  0.3000020117420486 \\
0.390625  0.30135857246073783 \\
0.39453125  0.30266923461713297 \\
0.3984375  0.3039340214688542 \\
0.40234375  0.30515275236431705 \\
0.40625  0.30632509731560575 \\
0.41015625  0.30745112142184805 \\
0.4140625  0.30853041487840527 \\
0.41796875  0.30956302232270566 \\
0.421875  0.31054866821318383 \\
0.42578125  0.311486883746898 \\
0.4296875  0.3123779526328245 \\
0.43359375  0.31322166512009036 \\
0.4375  0.31401775862876674 \\
0.44140625  0.3147662260508352 \\
0.4453125  0.3154668440512782 \\
0.44921875  0.3161198074864792 \\
0.453125  0.31672489297195416 \\
0.45703125  0.3172815739288376 \\
0.4609375  0.3177901296570452 \\
0.46484375  0.3182505792618319 \\
0.46875  0.3186626069310846 \\
0.47265625  0.31902643144294474 \\
0.4765625  0.3193419093261995 \\
0.48046875  0.31960872338866686 \\
0.484375  0.31982714130229145 \\
0.48828125  0.3199971624134442 \\
0.4921875  0.32011847545837546 \\
0.49609375  0.3201912036450008 \\
0.5  0.3202153776816157 \\
0.50390625  0.32019085385615403 \\
0.5078125  0.320117770543642 \\
0.51171875  0.31999610084070607 \\
0.515625  0.3198257312086899 \\
0.51953125  0.31960697997432075 \\
0.5234375  0.31933985354339794 \\
0.52734375  0.3190240844948979 \\
0.53125  0.3186599976018614 \\
0.53515625  0.31824773661700106 \\
0.5390625  0.3177870809182506 \\
0.54296875  0.31727833605389244 \\
0.546875  0.31672147849908117 \\
0.55078125  0.3161162210669548 \\
0.5546875  0.31546308295804837 \\
0.55859375  0.3147622852349567 \\
0.5625  0.31401363681344757 \\
0.56640625  0.3132173578968455 \\
0.5703125  0.3123734521578422 \\
0.57421875  0.31148218430845737 \\
0.578125  0.3105437653380415 \\
0.58203125  0.30955790584328235 \\
0.5859375  0.3085250799143672 \\
0.58984375  0.3074455731701647 \\
0.59375  0.30631935997174264 \\
0.59765625  0.30514683397270115 \\
0.6015625  0.3039279207306178 \\
0.60546875  0.3026629522074802 \\
0.609375  0.301352125714227 \\
0.61328125  0.2999954204474689 \\
0.6171875  0.2985933806849413 \\
0.62109375  0.2971459596292984 \\
0.625  0.29565349183920275 \\
0.62890625  0.2941161985296944 \\
0.6328125  0.29253427405338267 \\
0.63671875  0.29090821586439614 \\
0.640625  0.2892379486080412 \\
0.64453125  0.28752391104720554 \\
0.6484375  0.28576617497125656 \\
0.65234375  0.2839651777473686 \\
0.65625  0.2821212845596398 \\
0.66015625  0.2802346151450423 \\
0.6640625  0.27830556434975523 \\
0.66796875  0.2763342104503104 \\
0.671875  0.2743212402358167 \\
0.67578125  0.2722667885695638 \\
0.6796875  0.2701712157739193 \\
0.68359375  0.26803460375740396 \\
0.6875  0.2658573679052527 \\
0.69140625  0.2636398767643271 \\
0.6953125  0.2613824389546208 \\
0.69921875  0.25908551945023256 \\
0.703125  0.2567493705676013 \\
0.70703125  0.25437447169380906 \\
0.7109375  0.25196101118834485 \\
0.71484375  0.2495096063102901 \\
0.71875  0.24702050918926857 \\
0.72265625  0.24449412795697334 \\
0.7265625  0.2419306070736663 \\
0.73046875  0.23933056805578082 \\
0.734375  0.2366943728034801 \\
0.73828125  0.23402247233214096 \\
0.7421875  0.23131516664428262 \\
0.74609375  0.2285730418704233 \\
0.75  0.22579651729679984 \\
0.75390625  0.22298606164802573 \\
0.7578125  0.22014207310566514 \\
0.76171875  0.21726507722248267 \\
0.765625  0.2143554056580279 \\
0.76953125  0.21141341954282444 \\
0.7734375  0.2084396062118812 \\
0.77734375  0.20543458110698562 \\
0.78125  0.20239874962349774 \\
0.78515625  0.1993326108046093 \\
0.7890625  0.19623662108881154 \\
0.79296875  0.1931112645136943 \\
0.796875  0.18995693584249548 \\
0.80078125  0.1867742642153806 \\
0.8046875  0.18356372608614052 \\
0.80859375  0.18032578306606717 \\
0.8125  0.1770607814606488 \\
0.81640625  0.1737693491538136 \\
0.8203125  0.17045197004542967 \\
0.82421875  0.16710915979541388 \\
0.828125  0.16374151865662087 \\
0.83203125  0.16034965448078653 \\
0.8359375  0.15693398670339645 \\
0.83984375  0.15349493596746602 \\
0.84375  0.15003329042042624 \\
0.84765625  0.14654956196426305 \\
0.8515625  0.14304437422789792 \\
0.85546875  0.13951813032712818 \\
0.859375  0.13597145472563316 \\
0.86328125  0.13240472346754947 \\
0.8671875  0.12881862333975977 \\
0.87109375  0.12521399225546678 \\
0.875  0.12159117032321992 \\
0.87890625  0.11795052973157004 \\
0.8828125  0.11429268058049717 \\
0.88671875  0.11061853877025761 \\
0.890625  0.106928814918391 \\
0.89453125  0.10322372604275416 \\
0.8984375  0.09950392429474726 \\
0.90234375  0.09577017814552057 \\
0.90625  0.09202307129936109 \\
0.91015625  0.08826317474061092 \\
0.9140625  0.08449114391456937 \\
0.91796875  0.08070758106503849 \\
0.921875  0.07691319014210514 \\
0.92578125  0.07310861426016668 \\
0.9296875  0.06929449372418371 \\
0.93359375  0.06547149484378503 \\
0.9375  0.06164038520198179 \\
0.94140625  0.05780186710894561 \\
0.9453125  0.05395660874248716 \\
0.94921875  0.05010527144509804 \\
0.953125  0.04624864875039417 \\
0.95703125  0.04238755987285308 \\
0.9609375  0.0385228322080395 \\
0.96484375  0.03465531052452545 \\
0.96875  0.030785966450669566 \\
0.97265625  0.026915851012445906 \\
0.9765625  0.023046090037569206 \\
0.98046875  0.01917795280579386 \\
0.984375  0.015313052921924671 \\
0.98828125  0.011453553751085898 \\
0.9921875  0.007602535300507743 \\
0.99609375  0.00376505736299848 \\
1.0  0.0 \\
};
\end{axis}
\end{tikzpicture}%
 \\
		\input{fig/prob2/p1mlmc/grad.tex} \\
	\end{subfigure}
	\hspace{0.01\textwidth}
	\begin{subfigure}[t]{.35\textwidth}
		\vspace{-3.5cm}
		\centering
		\sethw{0.88}{0.88}
		\input{fig/prob2/p1mgopt/Ey.tex} \\
		\vspace{-0.25cm}
		\input{fig/prob2/p1mgopt/Vy.tex} \\
	\end{subfigure}
	\caption{Optimization results for Problem 2. The leftmost figures show the solution $u$ and its gradient $\nabla J(u)$ obtained using MG/OPT. The figures in the middle show the result obtained using finest level NCG only. The rightmost figures show the expected value and variance of the state $y$ at the optimal point. They look identical for both optimization methods. The figures are shown for the finest grid of $257 \times 257$.}
	\label{fig:results2}
\end{figure}

\subsection{Problem 3: Burgers' equation}
The constraint consists of the viscous Burgers' equation with a 1-dimensional spatial domain $D=(0,1)$. The initial value serves as the control variable.
\newcommand{\fd}{\mathcal{D}} 
\begin{align}
\begin{aligned}
\diffp{y}{t} &= \frac{s}{2}\diffp{y^2}{x} + \diffp{}{x} (k\diffp{y}{x})  && \mbox{on } D \times (0,T) \triangleq \mathcal{D},\\
y &= 0 && \mbox{on } \{0,1\} \times (0,T), \\
y(\cdot,0) &= u && \mbox{on } D.
\end{aligned}
\label{eq:p5pde}
\end{align}
We consider the setting $u\in L^2(D)$ and $y \in L^2(\mathcal{D}) \otimes L^2(\Omega)$.
The cost functional is
\begin{equation}
J(y,u) = \frac{1}{2}\mean{\int_D (y(x,T)-z(x))^2 \d{x}} + \frac{\alpha}{2}\int_D u^2 \d{x},
\end{equation}
with $z$ the target value at the final time $T=1$, which we define as
$$
z(x) = \begin{cases} 
\frac{1}{8}(1-\cos(5\pi x)) & \text{if } x \in [2/5,4/5], \\
0 & \text{otherwise}.
\end{cases}
$$
We take a constant convection coefficient $s=-1$. The diffusion coefficient $k$ is stochastic in space and constant in time. At any time instance, $k$ is a 1-dimensional lognormal field determined by the covariance (\ref{eq:covariance_Gaussian}) with $\lambda = 0.3$ and $\sigma^2 = 0.1$, multiplied by a scaling factor of $\num{1e-3}$. The diffusion coefficient is chosen small enough such that it does not completely dominate the behavior of the PDE. A sample of $k$ can be found in Figure \ref{fig:sample3}. 

\subsubsection{Optimality conditions}
Evaluating (\ref{eq:optcond_lagr_general}) for this specific problem leads to the optimality conditions
\begin{equation}
\left\{\begin{array}{rcll}
\diffp{y}{t} - \frac{s}{2}\diffp{y^2}{x} - \diffp{}{x} (k\diffp{y}{x}) & = & 0 & \quad \mbox{on } \mathcal{D}, \\[0.1cm]
-\diffp{p}{t} - sy\diffp{p}{x} + \diffp{}{x}(k\diffp{p}{x}) & = & 0 & \quad \mbox{on }, \fd\\[0.1cm]
\nabla \rJ(u) = \alpha u + \mean{p(\cdot,0)} & = & 0 & \quad \mbox{on } (0,L)
\end{array}\right.
\end{equation}
with $y(\cdot,0) = u$ and $p(\cdot,T) = z-y(\cdot,T)$ on $D$, and $p = y = 0$ on $\{0,1\} \times (0,T)$.

\subsubsection{Discretization and optimization details}
The PDE is solved using the MacCormack time stepping scheme \cite{maccormack1969effect}, which is second order accurate in both space and time. 
Let $\vec{y} = [y_0, \ldots, y_m] \in \mathbb{R}^m$ be the discretization with $m+1$ points of the state at a given time step and let $\vec{k} = [k_0, \ldots, k_m] \in \mathbb{R}^m$ be the discretization of the diffusion coefficient.
The scheme then advances $\vec{y}$ to the state $\vec{y}''= [y''_0, \ldots, y''_m] \in \mathbb{R}^m$ at the next time step as follows:
\begin{align*}
y'_i &= y_{i} + \frac{\Delta t}{\Delta x}(\psi_{i+1}-\psi_{i}) + \frac{k_i}{\Delta x^2}(y_{i+1}-2y_{i}+y_{i-1}), \\
y''_i &= \frac{1}{2}\left(y_{i} + y'_{i}
-\frac{\Delta t}{\Delta x}(\psi'_{i-1}-\psi'_{i}) + \frac{k_i}{\Delta x^2}(y'_{i+1}-2y'_{i}+y'_{i-1})
\right)
\end{align*}
for $i=1,\ldots,m-1$, with $\psi_{i} = \frac{1}{2}s_iy_{i}^2$, $\psi'_{i} = \frac{1}{2}{s_iy'_{i}}^2$, and $s_i = s = -1$. Furthermore, since they are pertaining to the zero boundary, $y_0 = y_{m} = y'_0 = y'_{m} = 0$.
For this experiment, $10001$ discretization points in time are used at all levels, i.e., the grid coarsening happens only in the spatial dimension. We use $K=L=4$ levels, the finest of which is $513 \times 10001$, the coarsest $33 \times 10001$. The MacCormack method is stable if
\begin{equation}
\label{eq:MacCormack_stab}
\Delta t \leq \frac{\Delta x^2}{\max(\vec{y})\Delta x + 2\max(\vec{k})}
\end{equation}
is satisfied for all time steps \cite{lewis2005}. The reader can check, for the finest level at least, that this condition holds by inspecting Figures \ref{fig:sample3} and \ref{fig:results3}. 
The level mapping operator from Problem 2 is reused here. 

Since the problem is nonlinear, and since the stability condition (\ref{eq:MacCormack_stab}) must remain guaranteed, performing the line searches in the NCG smoother is more involved. First we attempt to minimize a quadratic approximation of the cost function along the search direction, constructed using an additional gradient evaluation as explained in \S\ref{sec:opt/smoother}. 
If the resulting step size is negative, or so large that the new state calculation fails to satisfy the stability condition\footnote{Since the size of $y$ does not grow much in time (see Figure \ref{fig:results3}), one can expect that a successful check of (\ref{eq:MacCormack_stab}) for the starting value of $y$, i.e., for the control $u$, implies the satisfaction of (\ref{eq:MacCormack_stab}) for all time steps.}, or does not yield descent, it is discarded. Instead we then attempt backtracking, starting with $s$ small enough to satisfy the stability condition. We backtrack with a factor $4$ until Armijo's sufficient descent condition is satisfied. We observed that in the vast majority of NCG steps, the quadratic approximation yields an acceptable step size. Each NCG step thus still requires about 2 gradient evaluations.  

\subsubsection{Results}
The solution is shown in Figure \ref{fig:results3}. To reach a gradient norm of $\num{1e-4}$, Algorithm \ref{alg:MG/OPT} required 3 iterations that are detailed in Table \ref{tab:mgopt3} below.
A new sample calculation shows $J = \num{4.10e-04}$ and $\|\nabla J(u)\| = \num{7.56e-05}$. The total number of fine grid solves was $5252$, the total time was $5800$ s.
\begin{table}[h]
\begin{equation*}
\arraycolsep=3.0pt
\small
\begin{array}{l|lllllllllllll}
i & \epsilon^{(i)} & n_0 & n_1 & n_2 & n_3 & n_4 & J^{(i)}_0 & J^{(i)} &\norm{\gstart} & \norm{\gend} & \text{Solves} & \text{Time [s]}  \\
\hline
1 & \num{1.00e-01} & 96 & 32 & 16 & 8 & 4 & \num{4.69e-03} & \num{4.47e-04} & \num{6.65e-02} & \num{8.33e-04} & 190 & 225\\
2 & \num{5.00e-05} & 5479 & 532 & 188 & 77 & 18 & \num{4.47e-04} & \num{4.11e-04} & \num{8.51e-04} & \num{1.06e-04} & 3365 & 3717\\
3 & \num{5.00e-05} & 3472 & 209 & 56 & 19 & 4 & \num{4.11e-04} & \num{4.10e-04} & \num{1.05e-04} & \num{4.51e-05} & 1697 & 1858\\
\end{array}
\end{equation*}
\vspace{-0.4cm}
\caption{MG/OPT results for problem 3.}
\label{tab:mgopt3}
\end{table}

Table \ref{tab:ncg3} shows the results obtained by using finest level NCG.
After the $77$ iteration steps, an evaluation using new samples yields $J = \num{4.10e-4}$ and $\|\nabla J(u)\| = \num{8.46e-5}$. An equivalent of $34993$ fine grid solves were required, and the total time cost was $41669$ s.
\begin{table}[h]
\begin{equation*}
\arraycolsep=3.0pt 
\small
\begin{array}{l|llllllllll}
i & \epsilon^{(i)} & n_0 & n_1 & n_2 & n_3 & n_4 & J^{(i)} & \norm{\gend} & \text{solves} & \text{time[s]} \\
\hline
0 & \num{1.97e-03} & 64 & 32 & 16 & 8 & 4 & \num{4.69e-3} & \num{6.73e-02} & 56(\times 21) & 95(\times 21)\\
22 & \num{5.75e-4} & 3096 & 229 & 64 & 22 & 8 & \num{4.42e-4} & \num{5.71e-04} & 578(\times 44) & 673(\times 44)\\
65 & \num{5.00e-5} & 3281 & 242 & 84 & 34 & 10 & \num{4.11e-4} & \num{8.25e-5} & 645(\times 13) & 774(\times 13)\\
\end{array}
\end{equation*}
\vspace{-0.4cm}
\caption{Results for Problem 3 using finest level optimization only.}
\label{tab:ncg3}
\end{table}
\begin{figure}[h]
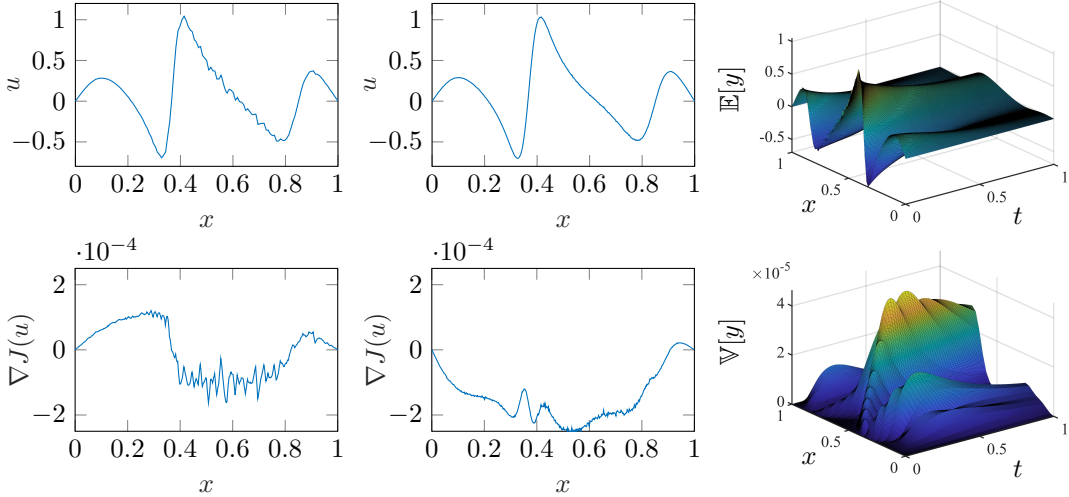
 
	\hspace{-0.01\textwidth}
	\begin{subfigure}[t]{.30\textwidth}
		\centering
		\sethw{0.5}{0.8}
		\input{fig/prob3/p1mgopt/u.tex} \\
		\input{fig/prob3/p1mgopt/grad.tex} \\
	\end{subfigure}
	\hspace{0.01\textwidth}
	\begin{subfigure}[t]{.30\textwidth}
		\centering
		\sethw{0.5}{0.8}
		\input{fig/prob3/p1mlmc/u.tex} \\
		\input{fig/prob3/p1mlmc/grad.tex} \\
	\end{subfigure}
	\hspace{0.01\textwidth}
	\begin{subfigure}[t]{.35\textwidth}
		\vspace{-3.5cm}
		\centering
		\sethw{0.88}{0.88}
		\input{fig/prob3/p1mgopt/Ey.tex} \\
		\vspace{-0.25cm}
		\input{fig/prob3/p1mgopt/Vy.tex} \\
	\end{subfigure}
	\caption{Optimization results for Problem 3. The leftmost figures show the solution $u$ and its gradient $\nabla J(u)$ obtained using MG/OPT. The figures in the middle show the result obtained using finest level NCG only. While the initial state looks different, they very quickly become indistinguishable when marching toward the final state.
	The expected value and variance of the state $y$ shown in the rightmost figures pertain to the solution obtained using MG/OPT. The figures are shown for the finest grid of $513 \times 10001$.}
	\label{fig:results3}
\end{figure}

\section{Conclusions}
\label{sec:conclusion}
A method was presented that uses the MLMC estimator for the calculation of gradients in the MG/OPT framework. Including optimization steps on the coarser levels can drastically reduce the number of iteration steps needed for a variety of problems. A theorem detailing the expected theoretical convergence was given and proven. Reaching a specified gradient tolerance is expected to be proportional in cost to estimating a gradient with that tolerance as the requested RMSE.

The results in this paper also show the importance of integrating the uncertainty in the optimization problem itself. 
The details of the estimation of statistical quantities should not remain constant, but vary as the optimization progresses. Only when an accurate estimation is required, should it be calculated. This is impossible if one deals with the uncertainties and the optimization separately. In our method this principle manifests itself in the use of fewer samples on coarser MG/OPT levels, and near the start of the optimization.

We showed that the ideas in this paper can be viable for several different problems, some benefiting more than others. The various PDE constraints and cost functions do not formally require any adaptations to the proposed strategy. 
Generally speaking, though, two requirements need to be fulfilled for the strategy to work. 
First, for the MLMC estimation to work, the discretization levels should be defined in a way that ensures that realizations for the same stochastic parameters on two consecutive discretization levels remain highly correlated. While the MLMC estimation of the gradient was already investigated in \cite{vanbarel2019robust} for the Laplace equation, the results in this paper show that MLMC can also work for other constraints, such as for the time dependent Burgers' equation. 
Second, it is required that MG/OPT performs well for a deterministic version of the problem. If it does not, one cannot hope for it to work in the more general stochastic case.

While not explicitly investigated here, MG/OPT, and, therefore, the strategy in this paper, can also deal with inequality constraints. The interested reader is referred to, e.g., \cite{ulbrich2011semismooth} for general information on numerically solving such problems, and to \cite{lewis2005} for details particular to the MG/OPT method.

\bibliographystyle{siamplain}
{\footnotesize
\bibliography{bib_avb}
}
\end{document}